\documentclass{amsart}

\usepackage{amsfonts}
\usepackage{amssymb}
\usepackage{amscd}
\usepackage{pictexwd,dcpic}
\usepackage{graphicx}
\usepackage{hyperref}
\hypersetup{
    colorlinks=true,
    linkcolor=black,
    filecolor=magenta, 
    urlcolor=cyan,
    citecolor=black,}
\usepackage[nameinlink,capitalize]{cleveref}
\usepackage{mathrsfs}
\usepackage{enumerate}

\newtheorem{theorem}{Theorem}[section]

\newtheorem{lemma}[theorem]{Lemma}

\newtheorem{corollary}[theorem]{Corollary}
\newtheorem{proposition}[theorem]{Proposition}

\theoremstyle{definition}

\newtheorem{setting}[theorem]{Setting}
\newtheorem{notation}[theorem]{Notation}

\newtheorem{remark}[theorem]{Remark}
\newtheorem{definition}[theorem]{Definition}
\numberwithin{equation}{section}
\newtheorem{example}[theorem]{Example}

\def\coker{\mathop{\rm coker}}

\def\sym{\mathop{\rm Sym}}
\def\fitt{\mathop{\rm Fitt}}
\def\max{\mathop{\rm max}}
\def\rk{\mathop{\rm rank}}
\def\quot{\mathop{\rm Quot}}

\def\spec{\mathop{\rm Spec}}

\def\hgt{\mathop{\rm ht}}
\def\dim{\mathop{\rm dim}}

\def\sup{\mathop{\rm sup}}
\def\gr{\mathop{\rm gr}}

\def\dep{\mathop{\rm depth}}
\def\grd{\mathop{\rm grade}}

\def\reg{\mathop{\rm reg}}

\def\p{\mathfrak{p}}

\def\a{\mathfrak{a}}
\def\m{\mathfrak{m}}
\def\n{\mathfrak{n}}
\def\tt{\mathfrak{t}}

\def\A{\mathcal{A}}
\def\D{\mathcal{D}}
\def\LL{\mathscr{L}}
\def\B{\mathcal{B}}

\def\rr{\mathcal{R}}
\def\ff{\mathcal{F}}

\def\N{\mathbb{N}}

\def\det{\operatorname{det}}

\def\mod{\operatorname{mod}}

\def\quot{\operatorname{Quot}}

\def\x{\underline{x}}

\def\a{\underline{a}}
\def\T{\underline{T}}

\def\fitt{{\rm Fitt}}
\def\rt{{\rm rt}}
\def\J{\mathcal{J}}
\def\H{\mathcal{H}}
\def\K{\mathcal{K}}

\def\BB{\mathcal{B}}  
\def\BL{\mathcal{L}}

\crefname{lemma}{lemma}{lemmas}
\Crefname{lemma}{Lemma}{Lemmas}
\crefname{observation}{observation}{observations}
\Crefname{observation}{Observation}{Observations}

\crefname{algorithm}{algorithm}{algorithms}
\Crefname{algorithm}{Algorithm}{Algorithms}

\crefname{setting}{setting}{settings}
\Crefname{setting}{Setting}{Settings}
\crefname{theorem}{theorem}{theorems}
\Crefname{theorem}{Theorem}{Theorems}
\crefname{proposition}{proposition}{propositions}
\Crefname{proposition}{Proposition}{Propositions}
\crefname{remark}{remark}{remarks}
\Crefname{Remark}{Remark}{Remarks}
\crefname{corollary}{corollary}{corollaries}
\Crefname{Corollary}{Corollary}{Corollaries}
\crefname{definition}{definition}{definitions}
\Crefname{definition}{Definition}{Definitions}
\crefname{notation}{notation}{notations}
\Crefname{notation}{Notation}{Notations}
\crefname{section}{section}{sections}
\Crefname{section}{Section}{Sections}

\title{On Rees Algebras of Ideals and Modules over Hypersurface Rings}

\author{
  Matthew Weaver
}
\address{Department of Mathematics, Purdue University, West Lafayette IN 47907, USA}
\email{weaver24@purdue.edu}
\date{}

\begin{document}
\maketitle

\begin{abstract}
The acquisition of the defining equations of Rees algebras is a natural way to study these algebras and allows certain invariants and properties to be deduced. In this paper, we consider Rees algebras of codimension 2 perfect ideals of hypersurface rings and produce a minimal generating set for their defining ideals. Then, using generic Bourbaki ideals, we study Rees algebras of modules with projective dimension one over hypersurface rings. We describe the defining ideal of such algebras and determine Cohen-Macaulayness and other invariants.

\end{abstract}

\section{Introduction}\label[section]{intro}
In this paper we study the defining ideal of the Rees algebra of an ideal or module. For $I$ an ideal of a Noetherian ring $R$, the Rees algebra of $I$ is the graded ring $\rr(I) = R[It] = R\oplus It \oplus I^2t^2\oplus \cdots$. Whereas this is an algebraic construction, much of the motivation to study these rings is geometric. Indeed, the Rees algebra $\rr(I)$ is the algebraic realization of the blowup of $\spec(R)$ along the subscheme $V(I)$. To study these rings, we remark there is a natural epimorphism of $R$-algebras $\Psi:R[T_1,\ldots,T_n] \rightarrow \rr(I)$ given by $\Psi(T_i)= \alpha_it$ where $I=(\alpha_1,\ldots,\alpha_n)$. The kernel $\J$ of this map is the \textit{defining ideal} of $\rr(I)$. This ideal and its generators, the \textit{defining equations} of $\rr(I)$, have been studied extensively under various conditions (see e.g. \cite{MU,KPU2,Morey,UV,KM,Vasconcelos,HSV1,Johnson,CHW,KPU1,BCS}).

This notion can be extended to Rees algebras of modules. If $E$ is an $R$-module with rank, the Rees algebra of $E$ is defined as the symmetric algebra $\sym(E)$ modulo its $R$-torsion submodule. It is important to realize this definition agrees with the previous one if $E$ is an $R$-ideal of positive grade. Once more, the motivation is geometric as these are the rings which occur in the process of successive blowups. As before, if one specifies a generating set of $E=Ra_1+\cdots+Ra_n$, there is a natural map $\Psi:R[T_1,\ldots,T_n] \rightarrow \rr(E)$ given by $\Psi(T_i)= a_i \in [\rr(E)]_1$ and the kernel of this map is the defining ideal of $\rr(E)$. Whereas Rees algebras of ideals have been studied to great extent, the treatment of Rees algebras of modules is still unclear in full generality. However, much progress has been made in recent years (see e.g. \cite{ReesAlgebrasOfModules,Costantini,CD,SUV1,Lin,EHU}). Many of the difficulties have been greatly reduced with the innovation of \textit{generic Bourbaki ideals}. With this, one is able to reduce the study of $\rr(E)$ and its defining ideal to the study of the Rees algebra of an ideal where information is much more readily available.

Despite the amount of previous work mentioned, most results where the defining equations are understood pertain to Rees algebras of $R$-ideals and $R$-modules for $R=k[x_1,\ldots,x_d]$, a polynomial ring over a field $k$. For this reason, we concern ourselves with quotients of a polynomials ring and remark there is strong geometric motivation to consider Rees algebras of ideals and modules in this setting. These rings will provide insight into blowups of more general schemes and rational maps between more general projective varieties. In particular, if $R$ is a graded quotient of $k[x_1,\ldots,x_d]$, it is interesting to consider the module of K\"{a}hler differentials $\Omega_k(R)$ and its Rees ring. The special fiber ring $\rr\big(\Omega_k(R)\big) \otimes k$ is exactly the homogeneous coordinate ring of the tangential variety of the projective variety with homogeneous coordinate ring $R$. 

We begin with Rees algebras of ideals and recall a few known results. In \cite{MU} Morey and Ulrich gave a complete description of the defining ideal of the Rees algebra of a linearly presented perfect ideal of grade 2 of $k[x_1,\ldots,x_d]$. In \cite{BM} Boswell and Mukundan extended this to the case where $I$ is again a perfect ideal of grade 2 of $k[x_1,\ldots,x_d]$, but is almost linearly presented. In other words, the presentation matrix of $I$ consists of linear forms with the exception of the last column which consists of homogeneous forms of the same degree $m\geq 1$.

We begin this paper by considering linearly presented ideals of a hypersurface ring $R=k[x_1,\ldots,x_{d+1}]/(f)$ where $f$ is a homogeneous polynomial of degree $m\geq 1$. By establishing similarities between linear presentation within hypersurface rings to almost linear presentation within polynomial rings, we are able to adapt the techniques of Boswell and Mukundan to this setting and generalize the result of Morey and Ulrich to Rees algebras of ideals in these rings. The main results regarding Rees algebras of ideals, \Cref{mainresult,depth}, are reformulated below.

\begin{theorem}\label{Theorem1}
Let $S=k[x_1,\ldots,x_{d+1}]$ be a polynomial ring over field $k$, $f\in S$ a homogeneous polynomial of degree $m$, and $R=S/(f)$. Let $I$ be a perfect $R$-ideal of grade 2 with presentation matrix $\varphi$ consisting of linear entries. Let $\overline{\,\cdot\,}$ denote images modulo $(f)$. If $I$ satisfies $G_d$, $I_1(\varphi) = \overline{(x_1,\ldots,x_{d+1})}$, and $\mu(I) = d+1$, then the defining ideal $\J$ of $\rr(I)$ satisfies
$$\J = \overline{\BL_m+(\det \BB_m)}$$
where the pair $(\BB_m,\BL_m)$ is the $m^{\text{th}}$ modified Jacobian dual iteration of $(B,\LL)$ for $B$ a modified Jacobian dual and $\LL=(\x\cdot B)$. Additionally, $\rr(I)$ is Cohen-Macaulay if and only if $m=1$ and is almost Cohen-Macaulay otherwise.
\end{theorem}

This result relies on the construction of the \textit{modified Jacobian dual} $B$ of a particular matrix $\psi$ with linear entries in the polynomial ring $S$. We then introduce the method of \textit{modified Jacobian dual iterations} to obtain the pair $(\BB_m,\BL_m)$ above and a minimal generating set of $\J$. We present this algorithm as a refinement of the one used by Boswell and Mukundan.

We now consider Rees algebras of modules and start by recalling some known results. By introducing generic Bourbaki ideals, in \cite{ReesAlgebrasOfModules} Simis, Ulrich, and Vasconcelos determined the defining equations of $\rr(E)$ for a module $E$ over $k[x_1,\ldots,x_d]$ of projective dimension one with linear presentation. Their proof relies on the fact that a generic Bourbaki ideal $I$ is a perfect ideal of grade 2 and $\rr(E)$ is a deformation of $\rr(I)$. With this, the structure of the defining ideal of $\rr(E)$ can be deduced from the defining ideal of $\rr(I)$ which is known from \cite{MU}. In \cite{Costantini} Costantini extended this by considering a module $E$ of projective dimension one with almost linear presentation. It was shown that the module $E$ admits a generic Bourbaki ideal $I$ which is almost linearly presented as well and there is a similar relation between $\rr(E)$ and $\rr(I)$. From there, Boswell and Mukundan's techniques were adapted to the module setting in order to produce the defining equations of $\rr(E)$.

Once more, we consider the situation where the ring $R$ is a hypersurface ring. Let $R= k[x_1,\ldots,x_{d+1}]/(f)$ and consider $E$, a linearly presented $R$-module with projective dimension one. We show that $E$ admits a generic Bourbaki ideal $I$ which is perfect of grade 2 in a hypersurface ring. Following the path laid out in \cite{Costantini} we show that, after a generic extension, $\rr(E)$ is a deformation of $\rr(I)$ and the defining ideals of these Rees rings are of a similar form. We then derive the defining equations of $\rr(E)$ from those of $\rr(I)$ which are known by \Cref{Theorem1}. A reformulation of the main result regarding Rees algebras of modules, \Cref{mainresultformodules}, is stated below.

\begin{theorem}
Let $S=k[x_1,\ldots,x_{d+1}]$ be a polynomial ring over field $k$ with $d\geq 2$, $f\in S$ a homogeneous polynomial of degree $m$, and $R=S/(f)$. Let $E$ be a finite $R$-module of projective dimension one with rank $e$ minimally generated by $n$ homogeneous elements of the same degree. Suppose the presentation matrix $\varphi$ of $E$ consists of linear entries and let $\overline{\,\cdot\,}$ denote images modulo $(f)$. If $E$ satisfies $G_d$, $I_1(\varphi) = \overline{(x_1,\ldots,x_{d+1})}$, and $n =d+e$, then the defining ideal $\J$ of $\rr(E)$ satisfies 
$$\J =\overline{\BL_m + (\det \BB_m)}$$
where the pair $(\BB_m,\BL_m)$ is the $m^{\text{th}}$ modified Jacobian dual iteration of $(B,\LL)$ for $B$ a modified Jacobian dual and $\LL=(\x\cdot B)$. Additionally, $\rr(E)$ is Cohen-Macaulay if and only if $m=1$ and is almost Cohen-Macaulay otherwise.
\end{theorem}

As before, this result relies on the construction of a modified Jacobian dual and the recursive algorithm of modified Jacobian dual iterations to produce a minimal generating set of the defining ideal. This result is nearly identical to the statement of \Cref{Theorem1} which is to be expected as modules of projective dimension one are the higher rank analog of perfect ideals of grade 2. Indeed, if $e=1$ the module $E$ is isomorphic to such an ideal and we recover the previous result.

We now describe how this paper is organized.

In \Cref{prelims} we provide the necessary background material on Rees algebras of ideals and modules. Here we introduce much of the terminology and usual conventions. Additionally, we review the properties and construction of generic Bourbaki ideals which will be essential for \Cref{modulesec}.

In \Cref{hypring} we begin the study of the Rees algebra of a linearly presented perfect ideal $I$ of grade 2 of a hypersurface ring and the defining ideal $\J$. We introduce a perfect ideal $J$ of grade 2 of a polynomial ring and relate the Rees algebras $\rr(I)$ and $\rr(J)$. Additionally, we introduce the modified Jacobian dual matrix.

In \Cref{iterationssec} we describe two recursive algorithms which produce equations of the defining ideal. The first algorithm is a slight adaptation of Boswell and Mukundan's technique and is referred to as the method of \textit{matrix iterations}. A second algorithm is then introduced as an improvement and is referred to as the method of \textit{modified Jacobian dual iterations}. Lastly, we give a condition for when the defining ideal agrees with the ideals produced by these methods.

In \Cref{defidealsec} we show that the defining ideal coincides with the ideal of modified Jacobian dual iterations when $I$ is an ideal with second analytic deviation one. Properties such as Cohen-Macaulayness and regularity of $\rr(I)$ are then studied.

Lastly, in \Cref{modulesec} we study the defining ideal $\J$ of $\rr(E)$ for a module $E$ of projective dimension one over a hypersurface ring. We produce a generic Bourbaki ideal $I$ and relate the defining ideal of $\rr(E)$ to that of $\rr(I)$ which is known from the previous section.

\section{Conventions, Notation, and Preliminaries}\label[section]{prelims}

We now provide the preliminaries and properties of Rees algebras of ideals and modules needed for this paper. Additionally, we describe the construction of the generic Bourbaki ideal.

\subsection{Rees Algebras of Ideals and Modules}

 As we will consider both Rees algebras of ideals and modules, we proceed as generally as possible in this section. For now assume $R$ is a Noetherian ring and let $E = Ra_1+\cdots +Ra_n$ be either a finitely generated $R$-module with positive rank or an ideal of positive grade. We make note of any unique conventions used specifically in either case. Recall there is a natural homogeneous epimorphism
$$R[T_1,\ldots,T_n] \longrightarrow \sym(E)$$
given by $T_i\mapsto a_i\in [\sym(E)]_1$ which induces the isomorphism
$$\sym(E) \cong R[T_1,\ldots,T_n]/\LL$$
where the kernel $\LL$ of the map above can be described from a presentation of $E$. Indeed, if $R^s\overset{\varphi}{\rightarrow}R^n \rightarrow E\rightarrow 0$ is any presentation of $E$, then $\LL$ is generated by the linear forms $\ell_1,\ldots,\ell_s$ where
$$[T_1\ldots T_n]\cdot \varphi = [\ell_1 \ldots \ell_s].$$
The map above induces another epimorphism
$$R[T_1,\ldots,T_n] \longrightarrow \rr(E)$$
by further factoring the $R$-torsion from $\sym(E)$. This induced map is the same as the one mentioned in the introduction and maps the generators of $E$ to their images in the first graded component of $\rr(E)$. This map induces an isomorphism
$$\rr(E) \cong R[T_1,\ldots,T_n]/\J$$
for some ideal $\J$ of $R[T_1,\ldots,T_n]$ which again is the \textit{defining ideal} of $\rr(E)$. Additionally, any member of a minimal generating set of $\J$ is called a \textit{defining equation} of $\rr(E)$. We see that $\LL\subseteq \J$, but note this containment is often strict. However, equality is possible in which case $E$ is said to be of \textit{linear type}. It is important to note that if $R$ is a standard graded ring, each of the maps and kernels above are bihomomegeous.

Recall the generators of $\LL$ can be obtained from a presentation matrix of $E$ and these are the linear equations of $\J$. Letting $R^s\overset{\varphi}{\rightarrow}R^n \rightarrow E\rightarrow 0$ be a presentation of $E$ and  $\ell_1,\ldots,\ell_s$ as before, there exists an $r\times s$ matrix $B(\varphi)$ with linear entries in $R[T_1,\ldots,T_n]$ such that
$$[T_1 \ldots T_n] \cdot \varphi=[\ell_1 \ldots \ell_s] = [a_1\ldots a_r]\cdot B(\varphi) $$
where $(a_1,\ldots,a_r)$ is an ideal containing the entries of $\varphi$. The matrix $B(\varphi)$ is called a \textit{Jacobian dual} of $\varphi$ with respect to $a_1,\ldots,a_r$. This is a source of higher degree generators of $\J$ as Cramer's rule guarantees $I_r(B(\varphi)) \subset\J$. There are often several choices for $B(\varphi)$, hence it is not unique in general. However, in the typical situation where $R=k[x_1,\ldots,x_d]$ and $I_1(\varphi) \subseteq (x_1,\ldots,x_d)$, the Jacobian dual $B(\varphi)$ with respect to $x_1,\ldots,x_d$ is unique if the entries of $\varphi$ are linear.

For $I$ an ideal, we say $I$ satisfies the condition $G_s$ if $\mu(I_\p) \leq \dim R_\p$ for all $\p \in V(I)$ with $\dim R_\p\leq s-1$. We note this condition can be rephrased in terms of Fitting ideals. Indeed if $\hgt I>0$, then $I$ satisfies $G_s$ if and only if $\hgt \fitt_j(I)\geq j+1$ for all $1\leq j\leq s-1$. More generally, if $E$ is a module of rank $e$, we say that $E$ satisfies $G_s$ if $\mu(E_\p) \leq \dim R_\p +e-1$ for all $\p \in \spec(R)$ with $1\leq \dim R_\p\leq s-1$ or equivalently, $\hgt \fitt_j(I)\geq j-e+2$ for all $e\leq j\leq s+e-2$. If $G_s$ is satisfied for all $s$, we say that the ideal or module satisfies $G_\infty$.

Now suppose $E$ is either an $R$-ideal or $R$-module once more. If $R$ is a local ring with maximal ideal $\m$ and residue field $k$, the \textit{special fiber ring} of $E$ is $\ff(E) = \rr(E)\otimes_R k \cong \rr(E)/\m \rr(E)$.  The Krull dimension $\ell(E) = \dim \ff(E)$ is the \textit{analytic spread} of $E$.

\subsection{Generic Bourbaki Ideals}

Generic Bourbaki ideals were introduced in \cite{ReesAlgebrasOfModules} as a method to study the Cohen-Macaulayness of Rees algebras of modules by reducing to the case of Rees algebras of ideals. We adopt the notation of \cite{ReesAlgebrasOfModules} and recall the constructions and definitions there. 


\begin{notation}\label[notation]{module notation}
Let $(R,\m)$ be a Noetherian local ring and $E$ a finite $R$-module with positive rank $e$. Let $U=Ra_1+\cdots+ Ra_n$ be a submodule of $E$ and consider a set of indeterminates 
$$Z=\big\{Z_{ij} \,\big| \, 1\leq i\leq n,\, 1\leq j\leq e-1\big\}.$$
Let $R' = R[Z]$ and $E' = E\otimes_R R'$. Now set $y_j = \sum_{i=1}^n Z_{ij}a_i$ for all $1\leq j\leq e-1$ and let $F' = \sum_{j=1}^{e-1} R'y_j$. Finally, let $R'' = R(Z) = R'_{\m R'}$ and with this, set $E'' = E\otimes_R R''$ and $F''=F'\otimes_{R'} R''$.
\end{notation}

\begin{proposition}[{\cite[3.2]{ReesAlgebrasOfModules}}]\label[proposition]{gbiassumptions}
Assume that $E$ is finite, torsion-free, and satisfies the $\tilde{G}_2$ condition, i.e. $E$ is free locally in depth one. If either $\grd (E/U) \geq 2$ or $U=E$, then $F'$ is a free $R'$-module of rank $e-1$ and $E'/F' \cong J$ where $J$ is an $R'$-ideal of positive grade. Additionally $E''/F'' \cong I$ where $I$ is an $R''$-ideal.
\end{proposition}

\begin{definition}[{\cite[3.3]{ReesAlgebrasOfModules}}]
We refer to the $R''$-ideal $I \cong E''/F''$ as a \textit{generic Bourbaki ideal} of $E$ with respect to $U$. In the case where $U=E$, $I$ is simply called a generic Bourbaki ideal of $E$.
\end{definition}

We remark that a generic Bourbaki ideal of $E$ with respect to submodule $U$ is essentially unique. Suppose $a_1,\ldots,a_n$ and $b_1,\ldots,b_m$ are generating sets of $U$ and let $Z$ and $Y$ denote the respective sets of indeterminates as in \Cref{module notation}. Let $I\subset R(Z)$ and $K\subset R(Y)$ denote the ideals constructed as in \cref{gbiassumptions}. There is an automorphism $\lambda$ of the $R$-algebra $S=R(Y,Z)$ and unit $u\in \quot(S)$ such that $\lambda(IS) = KSu$. Moreover, $u=1$ if $I$ and $K$ have grade at least 2 \cite[3.4]{ReesAlgebrasOfModules}.


Moreover, these ideals are not as enigmatic as they might seem. Given a presentation matrix of $E$, one can always obtain a presentation matrix of a generic Bourbaki ideal $I$. 

\begin{remark}[{\cite[p. 617]{ReesAlgebrasOfModules}}]\label[remark]{gbipresentation}
Let $(R,\m,k)$ be a Noetherian local ring and $E$ an $R$-module as in \Cref{gbiassumptions} and let $R^s\overset{\varphi}{\rightarrow}R^n \rightarrow E\rightarrow 0$ be a minimal presentation of $E$.
\begin{enumerate}[(a)]
    \item With $Z$ and $y_j$ as in \Cref{module notation}, after possibly multiplying $\varphi$ by an invertible matrix with entries in $k(Z)$, we may assume $\varphi$ presents $E''$ with respect to the generating set $y_1,\ldots,y_{e-1},a_e,\ldots,a_n$. Then
    $$\varphi= \begin{bmatrix} \hspace{2mm}A\hspace{2mm} \\
\varphi_I\end{bmatrix}$$
where $A$ and $\varphi_I$ are submatrices of sizes $e-1\times s$ and $n-e+1 \times s$ respectively. By construction, $\varphi_I$ is a presentation of $I$ and is minimal as $\mu(I) = n-e+1$.

\item If $R$ is standard graded and $\varphi$ consists of homogeneous entries of constant degree along each column, then $\varphi_I$ does as well. In particular, if $E$ is linearly presented then $I$ is as well.
\end{enumerate}
\end{remark}

To end this section, we provide the main tool to compare the Rees algebra of a module to the Rees algebra of its generic Bourbaki ideal.

\begin{theorem}[{\cite[3.5]{ReesAlgebrasOfModules}}]\label{passCMness}
Retain the assumptions of \cref{gbiassumptions}. Let $U$ be a reduction of $E$ and $I$ a generic Bourbaki ideal of $E$ with respect to $U$. Let $J \cong U''/F''$ as a subideal of $I$.
\begin{enumerate}[(a)]
    \item \begin{enumerate}[(i)]
        \item The Rees algebra $\rr(E)$ is Cohen-Macaulay if and only if $\rr(I)$ is Cohen-Macaulay.
        
        \item The module $E$ is of linear type and $\grd \rr(E)_+\geq e$ if and only if $I$ is of linear type.
    \end{enumerate}
    \item If either $(i)$ or $(ii)$ of $(a)$ hold, then $\rr(E'')/(F'') \cong \rr(I)$ and $y_1,\ldots,y_{e-1}$ is a regular sequence on $\rr(E'')$.
    
    \item If $\rr(E'')/(F'')\cong \rr(I)$, then $J$ is a reduction of $I$ with $r_J(I) = r_U(E)$. In particular, if $k$ is infinite and $U=E$, then $r(E) =r(I)$.
\end{enumerate}
\end{theorem}

Whenever the condition $\rr(E'')/(F'')\cong \rr(I)$ as in part $(b)$ holds, we say that $\rr(E'')$ is a \textit{deformation} of $\rr(I)$.

\section{Ideals of Hypersurface Rings}\label[section]{hypring}

In this section we consider the defining ideal $\J$ of $\rr(I)$ for $I$ a perfect ideal with grade 2 of $R$, a hypersurface ring. We introduce another grade 2 perfect ideal $J$ closely related to $I$ and then study the relation between $\rr(I)$ and $\rr(J)$.

\begin{setting}\label[setting]{setting1}
Let $S=k[x_1,\ldots,x_{d+1}]$ be a polynomial ring over a field $k$, $f\in S$ a homogeneous polynomial of degree $m\geq 1$, and $R= S/(f)$. Let $I = (\alpha_1,\ldots,\alpha_n)$ be a perfect $R$-ideal of grade 2 with a presentation matrix $\varphi$ consisting of homogeneous linear entries. Further assume that $I$ satisfies the $G_d$ condition and $n>d$.
\end{setting}

\begin{remark}\label{notli}
 We assume that $n > d$ in order to avoid the situation where $I$ is of linear type. Indeed if $n\leq d$, then $I$ satisfies the $G_\infty$ condition and is hence of linear type by \cite[2.6]{HSV1}. Additionally, for our purposes it is acceptable to assume that $m\geq 2$. If $m=1$, we reduce to the situation of \cite{MU} where the defining ideal is known. Regardless, we make no such assumption and retain this possibility. 
\end{remark}

\begin{notation}\label[notation]{notation1}
Let $\overline{\,\cdot\,}$ denote images  modulo the ideal $(f)$ and let $\psi$ be an $n\times n-1$ matrix consisting of linear entries in $S$ such that $\varphi = \overline{\psi}$. Let $\LL=(\ell_1,\ldots,\ell_{n-1},f)$ be the $S[T_1,\ldots,T_n]$-ideal where $[\ell_1 \ldots \ell_{n-1}]= [T_1 \ldots T_n]  \cdot \psi$.  
\end{notation}

We remark that $\psi$ is unique if $m\geq 2$ as the entries of this matrix are assumed to be linear. If $m=1$, then $\psi$ is not unique, but any such matrix can be chosen. We claim there exists an $S$-ideal which is perfect of grade 2 with $\psi$ as its presentation matrix. To see this, it suffices to show $\hgt I_{n-1}(\psi) \geq 2$ by the Hilbert-Burch theorem. Note that the image of this ideal in $R$ is exactly the corresponding ideal of minors of $\varphi$. As the height can only decrease passing to $R$ and $\hgt I_{n-1}(\varphi)=2$, the claim follows. Let $J$ denote this $S$-ideal and notice that since $I$ satisfies the $G_d$ condition, $J$ does as well. This follows by phrasing the condition $G_d$ in terms of Fitting ideals, relating the ideals of minors of $\psi$ with those of $\varphi$, and comparing their heights in a similar manner as before.

Notice that the $S[T_1,\ldots,T_n]$-ideal $(\ell_1,\ldots,\ell_{n-1})$ is  exactly the ideal defining $\sym(J)$. Moreover, $\overline{\LL}$ is the defining ideal of $\sym(I)$ as $\varphi = \overline{\psi}$ and with this, we see $S[T_1,\ldots,T_n]/ \LL \cong R[T_1,\ldots,T_n]/\overline{\LL} \cong \sym(I)$. Hence $\LL$ is a defining ideal of $\sym(I)$, in a sense, as a quotient of $S[T_1,\ldots,T_n]$. Thus there is a clear relation between the defining ideals of the two symmetric algebras as quotients of $S[T_1,\ldots,T_n]$ and it is natural to ask if there is a similar relation between the two Rees algebras $\rr(J)$ and $\rr(I)$. Before we explore this possibility, we give a description of the defining ideal $\J$ of $\rr(I)$ and produce an analogous ideal of $S[T_1,\ldots,T_n]$ which defines $\rr(I)$ as a quotient of this ring.

\begin{proposition}\label{Jasat}
With the assumptions of \Cref{setting1} and $\LL$ as in \Cref{notation1}, the defining ideal of $\rr(I)$ satisfies $\J = \overline{\LL:(x_1,\ldots,x_{d+1})^\infty}$.
\end{proposition}

\begin{proof}
Let $s\in \overline{(x_1,\ldots,x_{d+1})}$ and note that, as $I$ satisfies the condition $G_d$, locally $I_s$ satisfies $G_\infty$ as an $R_s$-ideal and is hence of linear type by \cite[2.6]{HSV1}. Thus $\J_s = \overline{\LL}_s$ and so there is some power $t$ such that $s^t \J \subset \overline{\LL}$, hence $\J \subseteq \overline{\LL}:\overline{(x_1,\ldots,x_{d+1})}^\infty$. However, we have $\overline{\LL}:\overline{(x_1,\ldots,x_{d+1})}^\infty \subseteq \J$ as $\overline{\LL} \subseteq \J$ and modulo $\J$, the image of $\overline{(x_1,\ldots,x_{d+1})}$ in $\rr(I)$ is an ideal of positive grade. 
\end{proof}

The claim regarding the grade of the image of $\overline{(x_1,\ldots,x_{d+1})}$ in $\rr(I)$ follows from the correspondence between associated primes of $R$ and $\rr(I)$ \cite{HS}.

\begin{setting}\label[setting]{setting1.5}
Adopt the assumptions of \Cref{setting1} and let $J$ be the $S$-ideal mentioned above. Let $\A = \LL:(x_1,\ldots,x_{d+1})^\infty$ in $S[T_1,\ldots,T_n]$ and assume that $J$ satisfies the condition $G_{d+1}$.
\end{setting}

We remark that this additional condition on $J$ is not a strong assumption. As previously mentioned, $J$ automatically satisfies the condition $G_d$ as $I$ does. Thus the assumption that $J$ satisfies $G_{d+1}$ is equivalent to $\hgt I_{n-d}(\psi) \geq d+1$.

\begin{remark}\label{diffdefideals}
Notice that $\A$ is the kernel of the composition
$$S[T_1,\ldots,T_n] \longrightarrow R[T_1,\ldots,T_n] \longrightarrow \rr(I)$$
where the first map quotients by $(f)$ and the second is the natural map. Similarly, $\LL$ is the kernel of the composition
$$S[T_1,\ldots,T_n] \longrightarrow R[T_1,\ldots,T_n] \longrightarrow \sym(I)$$
where the first map is the same as before and the second is the natural map onto the symmetric algebra.
\end{remark}


The saturation $\A$ is reminiscent of a similar ideal encountered in \cite{BM}. Observe that the generators of $\LL$ are all linear with respect to $x_1,\ldots,x_{d+1}$ with the exception of $f$. In \cite{BM} this same phenomenon occurred, but was due to a column of nonlinear entries in a presentation matrix. With this, we  proceed along a similar path to study $\A$, but must frequently take alternative approaches for the proofs presented here. In \cite{BM} the saturation encountered was a prime ideal, an assumption crucial for Boswell and Mukundan's arguments and one we do not necessarily possess unless $f$ is irreducible.

We now introduce a matrix associated to the generators of $\LL$ and the sequence $x_1,\ldots,x_{d+1}$ which will be an adaptation of the traditional Jacobian dual.

\begin{definition}
With the assumptions of \Cref{setting1.5} and letting $\LL$ and $\psi$ be as in \Cref{notation1}, we define a \textit{modified Jacobian dual} of $\psi$ to be the $d+1 \times n$ matrix $B=[B(\psi)\,|\,\partial f]$ where $B(\psi)$ is the Jacobian dual of $\psi$ with respect to $x_1,\ldots,x_{d+1}$ and $\partial f$ is a column with $f= [x_1 \ldots x_{d+1}]\cdot \partial f$. Here $|$ denotes the usual matrix concatenation. 
\end{definition}

Notice that the generators of $\LL$ are exactly the entries of the matrix product $[x_1 \ldots x_{d+1}]\cdot B$. We remark that any generating of the ideal of entries of $\psi$ could be used to construct $B(\psi)$, but not $B$ in general if $f$ is not contained in this ideal. This is no matter in our situation as $I_1(\psi) =(x_1,\ldots,x_{d+1})$. This follows as $I_1(\psi) \supset I_{n-d}(\psi)$ and the latter ideal has height $d+1$ since $J$ satisfies $G_{d+1}$.

The uniqueness of $B(\psi)$ is guaranteed as $\psi$ consists of linear entries of $S$, but the modified Jacobian dual $B$ is not necessarily unique as there are often multiple choices for $\partial f$. In particular, one may produce such a column $\partial f$ using partial derivatives as the notation suggests, when $k$ is a field of characteristic zero. In this case the degree $m$ is a unit and using the Euler formula, the entries of $\partial f$ can be taken as the partial derivatives of $\frac{1}{m}\cdot f$.

\begin{proposition}\label{hgtofL}
With the assumptions of \Cref{setting1.5}, $\hgt\LL = d+1$.
\end{proposition}

\begin{proof}
We show that $\dim \sym(I) =n$ which follows from the assumption that $I$ satisfies $G_d$ and the formula of Huneke and Rossi \cite[2.6]{HR}. Recall
$$\dim \sym(I) = \sup \big\{ \mu(I_\p) + \dim R/\p \hspace{1mm}\big|\hspace{1mm}\p\in \spec (R)\big\}$$
and to determine this supremum, we compute the value of $\mu(I_\p) + \dim R/\p$ in multiple cases. For a prime ideal $\p$, first note that if $\p \notin V(I)$ then $\mu(I_\p) =1$, hence $\mu(I_\p) + \dim R/\p \leq d+1$. If $\p\in V(I)$ and $\hgt \p< d$, then $\mu(I_\p) + \dim R/\p \leq \hgt \p +\dim R/\p =d$. Lastly, if $\p\in V(I)$ and $\hgt \p = d$, then $\mu(I_\p) + \dim R/\p = \mu(I) + \dim R/\p = n \geq d+1$. Hence $\dim \sym(I) =n$ and so $\hgt \LL=d+1$.
\end{proof}


\begin{proposition}\label{resint}
With the assumptions of \Cref{setting1.5} and $B$ a modified Jacobian dual of $\psi$, $\LL:(x_1,\ldots,x_{d+1}) = \LL+I_{d+1}(B)$.
\end{proposition}

\begin{proof}
Recall that $\LL = (\ell_1,\ldots,\ell_{n-1},f)$ where $[\ell_1\ldots \ell_{n-1}] = [x_1\ldots x_{d+1}]\cdot B(\psi)$ and with this, we claim that $(\ell_1,\ldots,\ell_{n-1}):(x_1,\ldots,x_{d+1})^\infty$ is a prime ideal of height $n-1$. Recall $J$ is a linearly presented perfect $S$-ideal of grade 2 satisfying $G_{d+1}$. As $(\ell_1,\ldots,\ell_{n-1})$ is the defining ideal of $\sym(J)$ it follows, in a similar manner to the proof of \Cref{Jasat}, that $(\ell_1,\ldots,\ell_{n-1}):(x_1,\ldots,x_{d+1})^\infty$ is the defining ideal of $\rr(J)$. As $\rr(J)$ is a domain of dimension $d+2$, indeed $(\ell_1,\ldots,\ell_{n-1}):(x_1,\ldots,x_{d+1})^\infty$ is a prime ideal of height $n-1$.

Now since $\psi$ consists of linear entries in $S$, $B(\psi)$ consists of linear entries in $k[T_1,\ldots,T_n]$. Hence by \cite[2.4]{BM}, $\hgt I_{d+1} (B(\psi)) = n-d-1$. With this and the fact that $x_1,\ldots,x_{d+1}$ is a regular sequence modulo $I_{d+1}(B(\psi))$, one has
$$\hgt \big((x_1,\ldots,x_{d+1}) +I_{d+1}(B(\psi))\big) = n.$$
Now as some power of $(x_1,\ldots,x_{d+1})$ multiples $\A=\LL:(x_1,\ldots,x_{d+1})^\infty$ into $\LL$ and hence into $\LL+I_{d+1}(B)$, it follows that any minimal prime of $\LL+I_{d+1}(B)$ contains either $(x_1,\ldots,x_{d+1})+ I_{d+1}(B(\psi))$ or $\A$. Thus $\hgt \big(\LL+I_{d+1}(B)\big) \geq n$ and since $\LL+I_{d+1}(B) \subset \LL:(x_1,\ldots,x_{d+1})$, we have $\hgt (\LL:(x_1,\ldots,x_{d+1})) \geq n$. Thus by \cite[1.5 and 1.8]{CU}, this ideal has height exactly $n$ and moreover $\LL+I_{d+1}(B) = \LL:(x_1,\ldots,x_{d+1})$.
\end{proof}


We turn our attention to the index of saturation of $\A=\LL:(x_1,\ldots,x_{d+1})^\infty$ and begin by stating a useful lemma.

\begin{lemma}[{\cite[3.5]{BM}}]\label{UlrichLemma}
Let $R$ be a Noetherian ring and $I$ a proper $R$-ideal. If $I^n \cap (0:I)=0$ for some $n\in \N$, then $I^n(0:I^\infty)=0$.
\end{lemma}

Note that if one were to consider an ideal of a complete intersection ring and its Rees algebra, the construction of the modified Jacobian dual and the previous results can be altered accordingly. For this reason, we provide a more general result on saturations. This result can also be obtained from the proof of \cite[6.1(a)]{KPU2}.

\begin{proposition}\label{GeneralIndexOfSat}
Let $R[x_1,\ldots,x_n]$ be a standard graded polynomial ring over Noetherian ring $R$ and let $f_1,\ldots, f_m \subset (x_1,\ldots, x_n)$ be homogeneous elements not necessarily of the same degree. Assume $m\geq n$ and $\deg f_1 \geq \deg f_2 \geq \cdots \geq \deg f_m$. Lastly, let $B$ denote an $n \times m$ matrix with $[f_1,\ldots,f_m] = [x_1,\ldots,x_n] \cdot B$. If $(f_1,\ldots,f_m):(x_1,\ldots, x_n) = (f_1,\ldots,f_m) +I_n(B)$, then one has the equality $(f_1,\ldots,f_m):(x_1,\ldots, x_n)^\infty =(f_1,\ldots,f_m):(x_1,\ldots, x_n)^N$ where the index $N$ is $N= \sum_{i=1}^{n}(\deg f_{i}-1)+1$.
\end{proposition}

\begin{proof}
Clearly $(f_1,\ldots,f_m):(x_1,\ldots, x_n)^N \subseteq (f_1,\ldots,f_m):(x_1,\ldots, x_n)^\infty$ so we need only exhibit the reverse containment. We claim that 
$$\big((x_1,\ldots, x_n)^N +(f_1,\ldots,f_m)\big )\cap \big((f_1,\ldots,f_m) + I_n(B)\big) = (f_1,\ldots,f_m)$$
by first noting that the ideals involved are homogeneous. Notice that $I_n(B)$ is generated by elements of degree at most $N-1$. Hence any element of this intersection of degree at least $N$ is contained in $(f_1,\ldots,f_m) + (x_1,\ldots, x_n)^N \cap I_n(B) \subset (f_1,\ldots,f_m) + (x_1,\ldots, x_n)I_n(B) \subset (f_1,\ldots,f_m)$. Also any element of this intersection of degree strictly smaller than $N$ is contained in $(f_1,\ldots,f_m)$. Now applying \Cref{UlrichLemma} to the image of $(x_1,\ldots, x_n)$ in $R[x_1,\ldots,x_n]/(f_1,\ldots,f_m)$, we find that 
$$\big((x_1,\ldots, x_n)^N +(f_1,\ldots,f_m) \big)\big((f_1,\ldots,f_m):(x_1,\ldots, x_n)^\infty\big)\subset (f_1,\ldots,f_m).$$
Hence $(f_1,\ldots,f_m):(x_1,\ldots, x_n)^\infty \subseteq (f_1,\ldots,f_m): \big((x_1,\ldots, x_n)^N +(f_1,\ldots,f_m) \big) = (f_1,\ldots,f_m):(x_1,\ldots, x_n)^N$ giving the required containment.
\end{proof}



\begin{corollary}\label{IndexOfSat}
With the assumptions of \Cref{setting1.5}, $\A=\LL:(x_1,\ldots,x_{d+1})^m$.
\end{corollary}



\subsection{The Relation between \texorpdfstring{$\rr(J)$}{R(J)} and \texorpdfstring{$\rr(I)$}{R(I)}}
Following the path laid out in \cite{KPU1}, we find a ring which maps onto $\rr(I)$ such that the kernel of this map is an ideal of height one and with this, we provide an alternative description of $\A$. In our situation, we will take such a ring to be the Rees algebra of $J$. Using the description and properties of the defining ideal of $\rr(J)$ from \cite{MU}, we study how these two Rees algebras relate to each other. 

Before we begin, we must update our notation with the biggest change being the meaning of $\overline{\,\cdot\,}$ to denote images. We adopt the following notation for the duration of this section and are careful to distinguish between the two interpretations of this symbol for the rest of the paper.

\begin{notation}\label[notation]{notation2}
Recall from \cite{MU} that, as $J$ is linearly presented and satisfies $G_{d+1}$, we have $\rr(J) \cong S[T_1,\ldots,T_n]/ \H$ where $\H =(\ell_1,\ldots,\ell_{n-1}) +I_{d+1}(B(\psi))$. Let $\overline{\,\cdot\,}$ denote images modulo $\H$ in $\rr(J)$. Additionally, let $B'$ be the $d\times n-1$ matrix obtained by deleting the last row of $B(\psi)$ and define the $S[T_1,\ldots,T_n]$-ideal $\K= (\ell_1,\ldots,\ell_{n-1}) +I_d(B')+(x_{d+1})$.
\end{notation}

\begin{proposition}\label{PropertiesOfA}
The ring $\rr(J)$ is a Cohen-Macaulay domain of dimension $d+2$ and the ideals $\overline{\K}$ and $\overline{(x_1,\ldots,x_{d+1})}$ are Cohen-Macaulay $\rr(J)$-ideals of height 1. Moreover, $\overline{(x_1,\ldots,x_{d+1})}$ is a prime ideal.
\end{proposition}

\begin{proof}
The statement that $\rr(J)$ is a domain of dimension $d+2$ follows easily from the fact that $S$ is a domain of dimension $d+1$ and $J$ is an ideal of positive height \cite{VasconcelosBook}. Moreover, $\rr(J)$ is Cohen-Macaulay by \cite[1.3]{MU}. Now as $J$ is a grade 2 perfect $S$-ideal satisfying $G_{d+1}$, its analytic spread is $\ell(J)=d+1$ \cite[4.3]{UV}. This is the dimension of the special fiber ring $\ff(J)$ which is a Cohen-Macaulay domain as $J$ is generated by homogeneous forms of the same degree \cite[proof of 2.4]{Morey}\cite{CGPU}. Thus $\overline{(x_1,\ldots,x_{d+1})}$ is indeed a Cohen-Macaulay prime ideal of height 1.

It remains to show $\overline{\K}$ is a Cohen-Macaulay $\rr(J)$-ideal of height 1. Notice that $\H$ has height $n-1$, hence the ideal $\K$ has height at least $n$ as it contains $\H$ and $x_{d+1} \in \K\setminus \H$. Observe that $\K$ can be written as $(\tilde{\ell}_1,\ldots,\tilde{\ell}_{n-1}) +I_d(B')+(x_{d+1})$ where $[\tilde{\ell}_1 \ldots \tilde{\ell}_{n-1}] = [x_1 \ldots x_d] \cdot B'$. Thus $(\tilde{\ell}_1,\ldots,\tilde{\ell}_{n-1}) +I_d(B')$ has height at least $n-1$, hence it has height exactly $n-1$, is Cohen-Macaulay, and $(\tilde{\ell}_1,\ldots,\tilde{\ell}_{n-1}) +I_d(B') = (\tilde{\ell}_1,\ldots,\tilde{\ell}_{n-1}):(x_1,\ldots,x_d)$ by \cite[1.5 and 1.8]{CU}. Thus $\K$ has height $n$ and is Cohen-Macaulay. Thus in $\rr(J)$, $\overline{\K}$ is a Cohen-Macaulay ideal of height 1. 
\end{proof}

\begin{proposition}\label{colons}
With $\K$ as in \Cref{notation2}, we have
$$\overline{(x_1,\ldots,x_{d+1})}^i = \overline{(x_1,\ldots,x_{d+1})}^{(i)}= (\overline{x_{d+1}}^i) :_{\rr(J)} \overline{\K}^{(i)}$$
and
$$\overline{\K}^{(i)} = (\overline{x_{d+1}}^i) :_{\rr(J)}  \overline{(x_1,\ldots,x_{d+1})}^{(i)}  $$
for all $i\in \N$.
\end{proposition}


\begin{proof}
First, by setting the degrees of the $x_i$ to 1 and the degrees of the $T_i$ to 0 temporarily, we see that $\gr_{\overline{(\x)}}{(\rr(J))} \cong \rr(J)$. As $\rr(J)$ is a domain, it follows that $\overline{(x_1,\ldots,x_{d+1})}^i = \overline{(x_1,\ldots,x_{d+1})}^{(i)}$. To show the proceeding equality, in $\rr(J)$ we have the containment $\overline{(x_1,\ldots,x_{d+1})}\overline{\K}\subseteq (\overline{x_{d+1}})$, hence $\overline{(x_1,\ldots,x_{d+1})}^i\overline{\K}^i\subseteq (\overline{x_{d+1}}^i)$. Now localizing at height one primes of $\rr(J)$, $\overline{(x_1,\ldots,x_{d+1})}^{(i)}\overline{\K}^{(i)}\subseteq (\overline{x_{d+1}}^i)$ and so $\overline{(x_1,\ldots,x_{d+1})}^{(i)}\subseteq (\overline{x_{d+1}}^i):\overline{\K}^{(i)}$. Writing $\K=(\tilde{\ell}_1,\ldots,\tilde{\ell}_{n-1}) +I_d(B')+(x_{d+1})$ as earlier, recall $(\tilde{\ell}_1,\ldots,\tilde{\ell}_{n-1}) +I_d(B') = (\tilde{\ell}_1,\ldots,\tilde{\ell}_{n-1}):(x_1,\ldots,x_d)$ and this is an ideal of height $n-1$. With this, we claim $0\neq I_d(B')\subset k[T_1,\ldots,T_n]$. If $I_d(B')=0$, then $(\tilde{\ell}_1,\ldots,\tilde{\ell}_{n-1}):(x_1,\ldots,x_d) = (\tilde{\ell}_1,\ldots,\tilde{\ell}_{n-1})$ and so $(x_1,\ldots,x_d)$ contains an element regular on $S[T_1,\ldots,T_n]/(\tilde{\ell}_1,\ldots,\tilde{\ell}_{n-1})$. However, this is impossible as one would also have $\hgt (\tilde{\ell}_1,\ldots,\tilde{\ell}_{n-1}) =n-1\geq d$. Now as $I_d(B')\neq 0$, we see $\overline{\K}\nsubseteq \overline{(x_1,\ldots,x_{d+1})}$ by degree considerations. Furthermore, as $\overline{(x_1,\ldots,x_{d+1})}$ is the unique associated prime of $\overline{(x_1,\ldots,x_{d+1})}^{(i)}$ it follows that $(\overline{x_{d+1}}^i):\overline{\K}^{(i)} \subseteq \overline{(x_1,\ldots,x_{d+1})}^{(i)}$.

A similar argument shows $\overline{\K}^{(i)} = (\overline{x_{d+1}}^i) :_{\rr(J)}  \overline{(x_1,\ldots,x_{d+1})}^{(i)}$.
\end{proof}

With this, we give a description of $\overline{\A}$ as an $\rr(J)$-ideal. Consider the fractional ideal $\D=\frac{\overline{f}\hspace{0.5mm}\overline{\K}^{(m)}}{\overline{x_{d+1}}^m}$ and note this is actually an $\rr(J)$-ideal as $f\in (x_1,\ldots,x_{d+1})^m$.

\begin{theorem}\label{DandA}
The $\rr(J)$-ideals $\D$ and $\overline{\A}$ are equal.
\end{theorem}

\begin{proof}
We begin by showing that $\D\subseteq \overline{\A}$. As mentioned, there is an induced surjective map of Rees algebras $\rr(J)\rightarrow \rr(I)$. Since $(x_1,\ldots,x_{d+1})\rr(I)$ is an ideal of positive grade, after a possible change of coordinates we may assume that the image of $x_{d+1}$ in $\rr(I)$ is a non-zerodivisor. With this, the image of $\D$ under this map annihilates an $\rr(I)$-regular element from which it follows that $\D$ is contained in the kernel, which is exactly $\overline{\A}$.

To show equality, we use the Cohen-Macaulayness of $\rr(J)$ and proceed in the same manner as that of \cite[3.10]{BM}. It is well known that in a Cohen-Macaulay ring, a proper ideal is unmixed of height one if and only if it satisfies Serre's condition $S_2$ as a module. As the $S_2$ condition is preserved under isomorphism and $\overline{\K}^{(m)}$ is unmixed of height one, it follows that $\D$ is as well since $\D \cong \overline{\K}^{(m)}$. Now in order to show $\D\subseteq \overline{\A}$ is actually equality, it suffices to show equality locally at the associated primes of $\D$, which are of height one.

If $\p \neq \overline{(x_1,\ldots,x_{d+1})}$ is such a height one prime ideal, then $\overline{\K}_\p = (\overline{x_{d+1}})_\p$. Hence $\D_\p = (f)_\p = \overline{\A}_\p$. Now suppose $\p =\overline{(x_1,\ldots,x_{d+1})}$ and notice that as $\overline{\K} \nsubseteq \overline{(x_1,\ldots,x_{d+1})}$,  we have $\overline{\K}^{(m)}_\p=\rr(J)_\p$. With this and \Cref{colons}, we see that $\overline{(x_1,\ldots,x_{d+1}})_\p = (\overline{x_{d+1}})_\p$. Additionally, observe that $\overline{\A}\nsubseteq \overline{(x_1,\ldots,x_{d+1})}$ hence $\overline{\A}_\p = \rr(J)_\p$ as well. With this, it suffices to show $\D_\p$ is the unit ideal. Notice that
$$\rr(J)_\p = \overline{\A}_\p=(\overline{f})_\p: \overline{(x_1,\ldots,x_{d+1})}^m_\p$$
and so $(\overline{f})_\p= \overline{(x_1,\ldots,x_{d+1})}^m_\p$ as $f\in(x_1,\ldots,x_{d+1})^m$. Thus $(\overline{f})_\p = (\overline{x_{d+1}})^m_\p$, from which we see $\D_\p =\frac{(\overline{x_{d+1}})^m_\p \overline{\K}^{(m)}_\p}{\overline{x_{d+1}}^m} = \rr(J)_\p$.
\end{proof}


We end this section by showing $m$ is actually the index of saturation of $\A$ in \Cref{IndexOfSat}.

\begin{proposition}\label{nissmallest}
With the assumptions of \Cref{setting1.5}, $m$ is the smallest integer such that $\A=\LL:(x_1,\ldots,x_{d+1})^m$.
\end{proposition}

\begin{proof}
Suppose there is a positive integer $i<m$ such that $\A=\LL:(x_1,\ldots,x_{d+1})^i$. With this, $(x_1,\ldots,x_{d+1})^i \A\subset \LL$ and so modulo $\H$ we have  $\overline{(x_1,\ldots,x_{d+1})}^i \overline{\A}\subset \overline{\LL} = (\overline{f})$. Now after localizing at $\overline{(\x)}=\overline{(x_1,\ldots,x_{d+1})}$  and noting once more that $\overline{\A}_{\overline{(\x)}}$ is the unit ideal, we have the containment $\overline{(x_1,\ldots,x_{d+1})}^i_{\overline{(\x)}} \subset (\overline{f})_{\overline{(\x)}}$. However, $i<m$ and so $f\in (x_1,\ldots,x_{d+1})^i$. Thus this containment is actually equality, $\overline{(x_1,\ldots,x_{d+1})}^i_{\overline{(\x)}} = (\overline{f})_{\overline{(\x)}}$.

Now by \Cref{IndexOfSat}, we know $\A=\LL:(x_1,\ldots,x_{d+1})^m$. Hence by proceeding as before, we find that $\overline{(x_1,\ldots,x_{d+1})}^m_{\overline{(\x)}} \subset (\overline{f})_{\overline{(\x)}}$ and since $f\in (x_1,\ldots,x_{d+1})^m$, this containment is equality as well. From this we obtain
$$ \overline{(x_1,\ldots,x_{d+1})}^i= \overline{(x_1,\ldots,x_{d+1})}^{(i)} = \overline{(x_1,\ldots,x_{d+1})}^{(m)} =\overline{(x_1,\ldots,x_{d+1})}^m$$
which is impossible.
\end{proof}

\section{Modified Jacobian Dual Iterations}\label[section]{iterationssec}

In this section we present two algorithms which produce new generators of $\A$ from preexisting ones. The first algorithm is inspired from the method used in Section 4 of \cite{BM} and in a similar manner we append columns to a particular matrix and take minors at each step. The second algorithm is a refinement of the first and is similar to the process used in \cite{CHW}. This alternative method also considers minors of a particular matrix at each step, but considers only a specified subset of minors. This is computationally more efficient than the first algorithm and will be the preferred method.



\subsection{Matrix Iterations}

We begin by generalizing the notion of the iterated Jacobian dual presented in \cite{BM} and repeat most of the constructions in a more general setting. For now let $R$ be a standard graded Noetherian ring and $B$ an $r\times s$ matrix with entries in $R$ of constant degree along each column. Let $\a= a_1, \ldots, a_r$ be a regular sequence of homogeneous elements of $R$ all of the same degree and let $L$ denote the ideal generated by the entries of the matrix product $\a\cdot B$. We define the matrix iterations of $B$ with respect to $\a$ as follows.

\begin{definition}\label[definition]{defit} Set $B_1 = B$ and $L_1=L$. Next, suppose the following pairs $(B_1, L_1), \ldots, (B_{i-1},L_{i-1} )$ have been constructed inductively such that for all $1\leq j\leq i-1$, $B_j$ is a matrix with $r$ rows consisting of homogeneous entries in $R$ of constant degree along each column and $L_j = (\a \cdot B_j)$. Now to construct the $i^{\text{th}}$ pair $(B_i,L_i)$, let 
$$L_{i-1} + \big(I_r(B_{i-1}) \cap (\a)\big) = L_{i-1}+(u_1,\ldots,u_l)$$
where $u_1,\ldots, u_l$ are homogeneous elements in $R$. There exists an $r\times l$ matrix $C$ consisting of homogeneous entries of constant degree along each column such that 
$$[u_1\ldots u_l] = [a_1\ldots a_r] \cdot C.$$
Now define $B_i$, an $i^{\text{th}}$ \textit{matrix iteration} of $B$, as
$$B_i = [B_{i-1} \,|\, C] $$
where $|$ denotes matrix concatenation. Lastly, set $L_i = (\a\cdot B_i)$.
\end{definition}

By construction, $B_{i-1}$ is a submatrix of $B_i$ and there is a containment of ideals $L_{i-1} \subseteq L_i$. Since the generators $u_1,\ldots,u_l$ above are not necessarily unique in each step, $B_i$ is not uniquely determined for $i>1$ and in general there are multiple candidates of different sizes for each matrix iteration. Despite the non-uniqueness of each matrix $B_i$, the ideal $L_i$ is well-defined which we show inductively. Certainly $L_1=L$ is well-defined as $L=(\a\cdot B)$. If we assume that $L_j$ is a well-defined ideal for all $1\leq j\leq i-1$, suppose that in the $i^{\text{th}}$ iterative step one has
$$L_{i-1}+(v_1,\ldots,v_t)=L_{i-1} + \big(I_r(B_{i-1}) \cap (\a)\big) = L_{i-1}+(u_1,\ldots,u_l)$$
for possibly different generating sets. Let $C$ and $C'$ denote matrices such that $\a\cdot C = [u_1,\ldots,u_l]$ and $\a\cdot C' = [v_1,\ldots,v_t]$. Then certainly $B_i = [B_{i-1}\,|\, C]$ and $B_i' =[B_{i-1}\,|\,C']$ are candidates for an $i^{\text{th}}$ matrix iteration of $B$. Regardless, by the above we have $(\a\cdot B_i) = (\a\cdot B_i')$, showing $L_i$ is indeed a well-defined $R$-ideal.

\begin{proposition}\label{LivsL}
 For all $i$, $L +I_r(B_i) = L_i+I_r(B_i)$.
\end{proposition}

\begin{proof}
Clearly $L +I_r(B_i) \subseteq L_i+I_r(B_i)$ so we need only exhibit the reverse containment. Write $L_i = L_{i-1} + (u_1,\ldots,u_l)$ following the notation of \Cref{modit}. Recall that this ideal is well-defined, hence we may take any generating set $u_1,\ldots,u_l$. Now notice that 
$$L_i =L_{i-1} +(u_1,\ldots,u_l)\subseteq L_{i-1}+I_r(B_{i-1})\subseteq L_{i-1}+I_r(B_{i}).$$
Repeating, it follows that $L_i\subseteq L_{i-1} + I_r(B_i)\subseteq  L_{i-2} + I_r(B_i)\subseteq \cdots \subseteq  L + I_r(B_i)$. 
\end{proof}


\begin{theorem}
The ideal $L+I_r(B_i)$ is uniquely determined for all $i$ by $L$ and the regular sequence $a_1,\ldots,a_r$.
\end{theorem}

\begin{proof}
We proceed by induction once again. Suppose that $B_1$ and $B_1'$ are two matrices satisfying $(\a\cdot B_1) =L = (\a\cdot B_1')$. As $\a=a_1,\ldots,a_r$ is a regular sequence, by \cite[4.4]{BM} $L+I_r(B_1) = L+I_r(B_1')$ which gives the initial step. Now suppose that $L+I_r(B_j)$ is a well-defined ideal for all $1\leq j\leq i-1$. If $B_i$ and $B_i'$ are two $i^{\text{th}}$ matrix iterations of $B$, then $(\a\cdot B_i) = L_i =(\a\cdot B_i')$. Now by \Cref{LivsL}, $L+ I_r(B_i) = L_i+I_r(B_i)$ and $L+ I_r(B_i') = L_i+I_r(B_i')$. Hence it suffices to show that $L_i+I_r(B_i) =L_i+I_r(B_i')$ and again this follows from \cite[4.4]{BM}.
\end{proof}

Note that as $R$ is a Noetherian ring, this procedure must eventually stabilize as the containments $L+I_r(B_i) \subseteq L+I_r(B_{i+1})$ produce an increasing chain of ideals. Additionally, notice that by repeatedly applying Cramer's rule we have $L+I_r(B_i) \subseteq L:(a_1,\ldots,a_r)^i$. Interesting as this may be, it is unclear if this containment is ever equality in general. Moreover, it is not clear if there is any relation between the stabilization points of $L + I_r(B_i)$ and $L:(a_1,\ldots,a_r)^i$.

\subsection{Ideals of height two in Hypersurface Rings}

As the notation suggests, we apply \Cref{defit} to the ideal $\LL$ and the sequence $x_1,\ldots,x_{d+1}$ to produce the matrix iterations of the modified Jacobian dual $B$. We state this below and explore some properties of the ideals produced. Immediately following, we present a refinement of \Cref{defit} as an alternative method to produce generators of $\A$.

\begin{notation}\label[notation]{notit}
With the assumptions of \Cref{setting1.5} and $B$ a modified Jacobian dual of $\psi$, let $B_i$ denote the $i^{\text{th}}$ matrix iteration of $B$ with respect to the sequence $x_1,\ldots,x_{d+1}$ as in \Cref{defit}.
\end{notation}

As previously mentioned, for all $i$ one has $\LL + I_{d+1}(B_i) \subseteq \LL:(x_1,\ldots,x_{d+1})^i$ and it is an interesting question if this containment is ever equality for some $i$. It is particularly interesting to ask if $\LL+I_{d+1}(B_m) = \LL:(x_1,\ldots,x_{d+1})^m$ as the latter ideal is $\A$ by \Cref{IndexOfSat}. Such a description of $\A$ is preferable as the generators of the ideal of matrix iterations are computed easily.

We now introduce an alternative procedure similar to the method of matrix iterations. Proceeding as in \Cref{defit}, in the creation of the $i^{\text{th}}$ matrix iteration of $B$, instead of considering all of the minors of $I_{d+1}(B_{i-1})\cap (x_1,\ldots,x_{d+1})$ we now consider only a subset of minors. These minors are the determinants of submatrices all of whose columns are columns of $B(\psi)$, except possibly for the last one which is some other column of $B_{i-1}$. In general, this leads to smaller matrices at each step and hence smaller ideals. Before we define this new algorithm, we introduce notation to ease the handling of this subset of minors.

\begin{notation}
Let $R$ be a ring and $M$ an $r\times s$ matrix with entries in $R$. If $M'$ is an $r\times t$ submatrix of $M$ for $t\leq s$, let $I_{r-i,i}(M',M)$ denote the $R$-ideal generated by the $r\times r$ minors of $M$ which are determinants of submatrices consisting of $r-i$ columns of $M'$ and some other $i$ columns of $M$. Notice that this is a subideal of $I_r(M)$ containing $I_r(M')$.
\end{notation}

We now introduce the method of modified Jacobian dual iterations of the pair $(B,\LL)$ with respect to the sequence $x_1,\ldots,x_{d+1}$. In \Cref{LivsL}, it was shown that the ideal of matrix iterations depends only on an updated matrix. In this new process, both an ideal and matrix must be altered at each step. 

\begin{definition}\label[definition]{modit}
 Set $\BB_1= B$ and $\BL_1 =\LL$. Next, suppose the following pairs $(\BB_1,\BL_1), \ldots, (\BB_{i-1},\BL_{i-1})$ have been constructed inductively such that for all $1\leq j\leq i-1$, $\BB_j$ is a matrix with $d+1$ rows of bihomogeneous elements of $S[T_1,\ldots,T_{d+1}]$ of constant bidegree along each column. To construct the $i^{\text{th}}$ pair $(\BB_i,\BL_i)$, let
$$\BL_{i-1} + \big(I_{d,1}(B(\psi),\BB_{i-1})\cap (x_1,\ldots,x_{d+1})\big) = \BL_{i-1} + (u_1,\ldots,u_l)$$
where $u_1,\ldots,u_l$ are bihomogeneous elements of $S[T_1,\ldots,T_n]$. Now there exists a matrix $C$ having bihomogeneous entries of constant bidegree along each column such that
$$[u_1\ldots u_l] = [x_1\ldots x_{d+1}]\cdot C.$$
Now take an $i^{\text{th}}$ \textit{modified Jacobian dual iteration} to be the pair $(\BB_i,\BL_i)$ where $\BB_i= [B(\psi)\,|\,C]$ and $\BL_i = \BL_{i-1} +(u_1,\ldots,u_l)$.
\end{definition}

Notice that the matrix $\BB_i$ is not necessarily unique at each step. However, by proceeding just as before, it can be seen that $\BL_i +I_{d,1}(B(\psi),\BB_i)$ is a well-defined ideal regardless of the choice of matrix $\BB_i$. As we consider a smaller set of minors at each step, we have $\BL_i +I_{d,1}(B(\psi),\BB_i) \subseteq \LL+I_{d+1}(B_i)\subseteq \A$. Eventually we will provide a criteria for when these ideals are equal, but first we must introduce an alternative description of the ideal of modified Jacobian dual iterations.

\begin{theorem}\label[theorem]{equal}
With the assumptions of \Cref{setting1.5} and $\K$ as in \Cref{notation2}, one has $\frac{\overline{f} \,\overline{\K}^m}{\overline{x_{d+1}}^m} = \overline{ \BL_{m}+I_{d,1}(B(\psi),\BB_m)}$ in $\rr(J)$.
\end{theorem}

\begin{proof}
Letting $D_i = \frac{\overline{f} \,\overline{\K}^i}{\overline{x_{d+1}}^i}$ and  $D_i' =\overline{ \BL_{i}+I_{d,1}(B(\psi),\BB_i)}$, it is clear that $D_i \subseteq D_{i+1}$ and $D_i' \subseteq D_{i+1}'$ for any $i$. We show $D_i = D_i'$ for all $1\leq i\leq m$ by induction.


First suppose that $i=1$ and we begin by showing $D_1\subseteq D_1'$. Notice $D_1' =\overline{ \LL + I_{d+1}(B)}$ in this case. Recall from the proof of \Cref{PropertiesOfA}, $\K$ may be written as $\K=(\tilde{\ell}_1,\ldots,\tilde{\ell}_{n-1}) +I_d(B')+(x_{d+1})$ where $[\tilde{\ell}_1 \ldots \tilde{\ell}_{n-1}] = [x_1 \ldots x_d] \cdot B'$. Thus modulo $\H$ we see $\overline{(\tilde{\ell}_1,\ldots,\tilde{\ell}_{n-1})} \subset (\overline{x_{d+1}})$ and so $\frac{\overline{f} \,\overline{(\tilde{\ell}_1,\ldots,\tilde{\ell}_{n-1})}}{\overline{x_{d+1}}}\subseteq (\overline{f}) \subseteq \overline{\LL}$. Now let $w\in I_d(B')$ and for convenience, assume that $w$ is the determinant of the submatrix consisting of the first $d$ columns of $B'$. Let $M$ be the $d+1\times d+1$ submatrix of $B$ consisting of the first $d$ columns of $B(\psi)$ and the last column  of $B$, $\partial f$. Now by Cramer's rule, in $\rr(J)$ we have $\overline{x_{d+1}}\cdot\overline{\det M}  = \overline{f}\cdot \overline{w}$. Thus $\frac{\overline{f} \overline{w}}{\overline{x_{d+1}}}= \overline{\det M} \in \overline{I_{d+1}(B)}$, hence $D_1\subseteq D_1'$.

To show the reverse containment, recall that $x_{d+1} \in \K$ and so $\overline{f} = \frac{\overline{f}\overline{x_{d+1}}}{\overline{x_{d+1}}} \in D_1$ from which it follows that $\overline{\BL_1} =\overline{\LL} \subset D_1$. Let $m \in I_{d,1}(B(\psi),\BB_1)=I_{d+1}(B)$ and note this ideal contains $I_{d+1}(B(\psi))$. If $m\in I_{d+1}(B(\psi))$ then $\overline{m} =0$ in $\rr(J)$ and there is nothing to be shown. Thus we may assume $m \in I_{d+1}(B)\setminus I_{d+1}(B(\psi))$ and for convenience, take $m=\det M$ where $M$ is the submatrix of $B$ consisting the first $d$ columns of $B(\psi)$ and the last column $\partial f$. If $w$ denotes the determinant of the submatrix of $B'$ consisting of the first $d$ columns of $B'$, then by Cramer's rule $\overline{x_{d+1}} \cdot \overline{\det M} = \overline{f}\cdot \overline{w}$. Hence $\overline{m} = \frac{\overline{f}\overline{w}}{\overline{x_{d+1}}} \in D_1$ and so $D_1'\subseteq D_1$ and the initial claim follows.

Now assume $m\geq 2$ and $D_i=D_i'$ for all $1\leq i\leq m-1$ and we first show that $D_m \subseteq D_m'$. Consider $\frac{\overline{f} \overline{w_1}\cdots \overline{w_m}}{\overline{x_{d+1}}^m} \in D_m$ for $w_1,\ldots, w_m \in \K$. Let $\overline{w'} = \frac{\overline{f} \overline{w_1}\cdots \overline{w_{m-1}}}{\overline{x_{d+1}}^{m-1}}$ and note that $\overline{w'} \in D_{m-1} = D_{m-1}'$ by the induction hypothesis. With this, we show that $\frac{\overline{f} \overline{w_1}\cdots \overline{w_m}}{\overline{x_{d+1}}^m } = \frac{\overline{w'}\overline{w_m}}{\overline{x_{d+1}}}$ is contained in $D_m'$. If $w' \in \BL_{m-1}$, then $\overline{w'} \in D'_{m-2}= D_{m-2}$ if $m>2$ and $\overline{w'} \in (\overline{f})$ if $m=2$. In either case, $\frac{\overline{w'}\overline{w_m}}{\overline{x_{d+1}}} \in D_{m-1} = D_{m-1}' \subseteq D_m'$. Now suppose $w' \in I_{d,1}(B(\psi),\BB_{m-1})$ and as before, recall this ideal contains $I_{d+1}(B(\psi))$. If $w'$ is pure in the variables $T_1,\ldots,T_n$, then $w'\in I_{d+1}(B(\psi))$ and so $\overline{w'} =0$ and $\frac{\overline{w'}\overline{w_m}}{\overline{x_{d+1}}} =0$ as $\rr(J)$ is a domain. Thus we may assume that $w' \in I_{d,1}(B(\psi),\BB_{m-1})\cap (x_1,\ldots,x_{d+1}) =(u_1,\ldots,u_l)$ following the notation of \Cref{modit}. It suffices to show that $\frac{\overline{u_p}\overline{w_m}}{\overline{x_{d+1}}} \in D_m'$ for all $1\leq p\leq l$. Without loss of generality, we may assume $w'=u_p$ for some $p$ in the range above. Write 
$$w' = \sum_{k=1}^{d+1}w_k' x_k$$
for $w_k'\in S[T_1,\ldots,T_{d+1}]$. Now if $\overline{w_m}\in (\overline{x_{d+1}}) \subseteq \overline{\K}$, then $\frac{\overline{w'}\overline{w_m}}{\overline{x_{d+1}}} \in D_{m-1}' \subseteq D_m'$. Thus we may assume $w_m \in I_d(B')$ and without loss of generality, we may assume that $w_m$ is the determinant of the submatrix consisting of the first $d$ columns of $B'$. Now let $M$ be the $d+1 \times d$ submatrix consisting of the first $d$ columns of $B(\psi)$. By \Cref{crlemma}, $\overline{x_k}\cdot \overline{w_m} = (-1)^{k-d-1} \overline{x_{d+1}}\cdot\overline{w_{m_k}}$ in $\rr(J)$ where $w_{m_k}=\det M_k$ and $M_k$ is the submatrix of $M$ obtained by deleting the $k^{\text{th}}$ row. With this, we have
$$\frac{\overline{w'}\overline{w_m}}{\overline{x_{d+1}}} =\frac{\sum_{k=1}^{d+1}\overline{w_k'} \overline{x_k}\overline{w_m}}{\overline{x_{d+1}}}= \sum_{k=1}^{d+1} (-1)^{k-d-1}\overline{w_{m_k}} \overline{w_k'}$$
and note this last sum is exactly the determinant of the $d+1\times d+1$ matrix $[M\,|\, C]$ modulo $\H$ where $C$ is the column with entries $w_k'$ for $k=1,\ldots,d+1$. We remark that as it was assumed $w'=u_p$, $C$ is exactly the column corresponding to $u_p$ in the creation of $\BB_m$ following \Cref{modit}. Hence we have $\frac{\overline{w'}\overline{w_m}}{\overline{x_{d+1}}}\in \overline{I_{d+1}(\BB_m)} \subseteq D_m'$, giving the containment $D_m\subseteq D_m'$. 

To show the reverse containment, note that $\BL_{m} \subseteq \BL_{m-1}+ I_{d,1}(B(\psi),\BB_{m-1})$, hence $\overline{\BL_m} \subseteq D_{m-1}' =D_{m-1} \subseteq D_m$ by the induction hypothesis. Let $w\in I_{d,1}(B(\psi),\BB_{m})$ and as before we may assume $w\notin I_{d+1}(B(\psi))$. Thus
without loss of generality, assume that $w$ is the determinant of the submatrix of $\BB_m$ consisting of the the first $d$ columns of $B(\psi)$ and a column of $\BB_m$ corresponding to some $u_p$ where $1\leq p\leq l$ and $I_{d,1}(B(\psi),\BB_{m-1})\cap (x_1,\ldots,x_{d+1}) =(u_1,\ldots,u_l)$. If $w'$ denotes the determinant of the submatrix of $B'$ consisting of the first $d$ columns of $B'$, then by Cramer's rule $\overline{x_{d+1}}\cdot \overline{w} = \overline{w'}\cdot \overline{u_p}$ and so $\overline{w} = \frac{\overline{w'}\cdot \overline{u_p}}{\overline{x_{d+1}}}$. Now $\overline{u_p} \in \overline{I_{d,1}(B(\psi),\BB_{m-1})} \subset D_{m-1}' = D_{m-1}$ by the induction hypothesis, from which it follows that $\overline{w} \in D_m$. Hence $D_m'\subseteq D_m$ and so the two $\rr(J)$-ideals are equal.
\end{proof}

\begin{lemma}[{\cite[4.3]{BM}}]\label[lemma]{crlemma}
Let $R$ be a commutative ring, $[a_1\ldots a_r]$ a $1\times r$ matrix, and $M$ an $r\times r-1$ matrix with entries in $R$. For $1\leq t\leq r$, let $M_t$ be the $r-1\times r-1$ submatrix of $M$ obtained by deleting the $t^{\text{th}}$ row of $M$ and set $m_t=\det M_t$. Then in the ring $R/(\a\cdot M)$
$$\overline{a_t}\cdot\overline{m_k} = (-1)^{t-k} \overline{a_k}\cdot \overline{m_t}$$
for all $1\leq t\leq r$ and $1\leq k\leq r$.
\end{lemma}

With the description of the ideal of modified Jacobian dual iterations given in \Cref{equal}, we now have the following criterion for when either iterative method yields a complete generating set of $\A$.

\begin{corollary}\label[corollary]{powersym}
If $\overline{\K}^m = \overline{\K}^{(m)}$, then $\A = \BL_m + I_{d,1}(B(\psi),\BB_m)= \LL+I_{d+1}(B_m)$.
\end{corollary}

\begin{proof}
Recall that $\A = \LL:(x_1,\ldots,x_{d+1})^m = \frac{\overline{f} \,\overline{\K}^{(m)}}{\overline{x_{d+1}}^m}$ by \Cref{IndexOfSat} and \Cref{DandA}. Now from \Cref{equal} and the containments
$$\frac{\overline{f} \,\overline{\K}^m}{\overline{x_{d+1}}^m} = \overline{\BL_m + I_{d,1}(B(\psi),\BB_m)} \subseteq \overline{ \LL+I_{d+1}(B_m)} \subseteq \overline{\A} = \frac{\overline{f} \,\overline{\K}^{(m)}}{\overline{x_{d+1}}^m}
$$
the claim follows.
\end{proof}

It is interesting to note that the result above shows the indices of stabilization of the two algorithms are equal in this setting and moreover, they both agree with the index of stabilization of $\A$ as a saturation. Now that we a have sufficient condition for when $\A$ is equal to the ideals produced by the two iterative methods, we investigate when this criterion is satisfied.

\section{Ideals of Second Analytic Deviation One}\label[section]{defidealsec}

In this section we consider Rees algebras of ideals with second analytic deviation one. We present a minimal generating set of the defining ideal in terms of the modified Jacobian dual iterations. Additionally, we investigate properties of the Rees algebra such as Cohen-Macaulayness, depth, and regularity.

Recall that the second analytic deviation of $I$ is $\mu(I) -\ell(I)$ and for the duration of this section we assume this is one in addition to the assumptions of \Cref{setting1.5}. By \cite{UV}, $I$ is of maximal analytic spread $\ell(I) =d$, hence this additional assumption is equivalent to $n=d+1$. With this, by \cite[2.6]{HSV1} $J$ is of linear type, hence $\rr(J) \cong \sym(J)$ and this is a complete intersection domain. This can also be seen from \cite{MU} and \Cref{notation2} as $I_{d+1}(B(\psi)) =0$ hence $\H = (\ell_1,\ldots,\ell_d)$ in this setting.

\begin{proposition}\label{Kscm}
With $n=d+1$ and $\K$ as in \Cref{notation2}, $\overline{\K}$ is generically a complete intersection and is a strongly Cohen-Macaulay ideal of $\rr(J)$.
\end{proposition}


\begin{proof}
 Let $\p$ be a prime ideal of height one containing $\overline{\K}$. The proof of \Cref{colons} shows $\overline{(x_1,\ldots,x_{d+1})}$ is not an associated prime of $\overline{\K}$,  hence $(\overline{x_{d+1}})_\p : \overline{\K}_\p = \overline{(x_1,\ldots,x_{d+1})}_\p = \rr(J)_\p$. With this we see $\overline{\K}_\p \subseteq (\overline{x_{d+1}})_\p$, hence $\overline{\K}_\p= (\overline{x_{d+1}})_\p$ showing $\overline{\K}$ is generically a complete intersection.

Now notice that as $n=d+1$, $\overline{\K} = (\overline{w},\overline{x_{d+1}})$ where $w= \det B'$. Recall that $\overline{\K}$ is an ideal of height one, hence it is an almost complete intersection ideal. Additionally by \Cref{PropertiesOfA}, $\overline{\K}$ is a Cohen-Macaulay ideal, hence by \cite[2.2]{Huneke2} $\overline{\K}$ is a strongly Cohen-Macaulay $\rr(J)$-ideal. 
\end{proof}

\begin{proposition}\label{hgtIdB}
With the assumptions of \Cref{setting1.5} and $n=d+1$, one has $\hgt I_{d}(B(\psi)) =2$ in the ring $S[T_1,\ldots, T_{d+1}]$.
\end{proposition}

\begin{proof}
As mentioned, the additional assumption that $n=d+1$ implies $\rr(J)\cong \sym(J)$ and this is a complete intersection domain of dimension $d+2$. Furthermore, since $J$ is linearly presented, $B(\psi)$ consists of linear entries in $k[T_1,\ldots,T_{d+1}]$. Thus there is an isomorphism of symmetric algebras $\sym(J) \cong \sym_{k[\T]}(E)$  where $E=\coker B(\psi)$. Since $\sym(J)$ is a domain, by \cite[6.8]{HSV2} we have
$$d+2 = \dim \sym(J) = \dim {\rm sym}_{k[\T]}(E) = \rk E + \dim k[T_1,\ldots,T_{d+1}].$$
Hence $\rk E =1$ and so by \cite[6.8 and 6.6]{HSV2}, $\grd I_d(B(\psi)) \geq 2$. Now as $S[T_1,\ldots,T_{d+1}]$ is Cohen-Macaulay and this is the largest height possible, the claim follows.
\end{proof}

We now present the main result of the section and temporarily return to the setting of \Cref{notation1} for its statement. We remark that as $n=d+1$, the modified Jacobian dual $B$ is a square matrix. Hence, following the construction laid out in \Cref{modit}, each modified Jacobian dual iteration $\BB_i$ is a square matrix as well. 

\begin{theorem}\label{mainresult}
Let $S=k[x_1,\ldots,x_{d+1}]$, $f\in S$ a homogeneous polynomial of degree $m$, and $R=S/(f)$. Let $I$ be a perfect $R$-ideal of grade 2 with linear presentation matrix $\varphi$. If $I$ satisfies $G_d$, $I_1(\varphi) = \overline{(x_1,\ldots,x_{d+1})}$, and $\mu(I) = d+1$ then the defining ideal $\J$ of $\rr(I)$ satisfies 
$$\J =\overline{\LL+I_{d+1}(B_m)} = \overline{\BL_m + (\det \BB_m)}$$
where $\overline{\,\cdot\,}$ denotes images modulo $(f)$. Additionally,  $\ff(I) \cong k[T_1,\ldots,T_{d+1}]/(\mathfrak{f})$ where $\deg \mathfrak{f} = md$.
\end{theorem}

We remark that the assumption $I_1(\varphi) = \overline{(x_1,\ldots,x_{d+1})}$ ensures that $J$ satisfies $G_{d+1}$. This condition implies $I_1(\psi) = (x_1,\ldots,x_{d+1})$ when $m\geq 2$ and if $m=1$, then $\psi$ is not unique, but can be chosen to have this ideal of entries. Now as $\mu(J)=d+1$, this is exactly the $d^{\text{th}}$ Fitting ideal of $J$ and the discussion prior to \Cref{diffdefideals} confirms $J$ satisfies $G_{d+1}$. Thus the conditions of \Cref{setting1.5} are met.


\begin{proof}
We return to the setting of \Cref{notation2} and let $\overline{\,\cdot\,}$ denote images in $\rr(J)$ once again. By \Cref{powersym}, it suffices to show $\overline{\K}^m = \overline{\K}^{(m)}$ and we begin by showing that
$$\mu(\overline{\K}_\p) \leq \hgt \p -1 =1$$
for any prime ideal $\p\in V(\overline{\K})$ with $\hgt \p =2$. Let $\p$ be such a prime ideal of $\rr(J)$ and first note that if $\p \nsupseteq \overline{(x_1,\ldots,x_{d+1})}$, then $\overline{\K}_\p = (\overline{x_{d+1}})_\p$ and the claim is satisfied. Now assume that $\p\supseteq \overline{(x_1,\ldots,x_{d+1})}$ and recall that $\hgt I_d(B(\psi)) = 2$. With this, the ideal $(x_1,\ldots,x_{d+1})+I_d(B(\psi))$ is of height $d+3$ in $S[T_1,\ldots,T_{d+1}]$. Thus the image of this ideal in $\rr(J)$ is of height $3$, hence $\p \nsupseteq \overline{I_d(B(\psi))}$. Now let $w = \det B'$ and write $w_1,\ldots,w_d$ for the other $d\times d$ minors where each $w_i$ is the determinant of the submatrix of $B(\psi)$ obtained by deleting row $i$. 

Since $\overline{w} \in \overline{\K} \subset \p$ and $\p \nsupseteq \overline{I_d(B(\psi))}$, it follows that $\overline{w_j} \notin \p$ for some $j$. By \Cref{crlemma}, we have $\overline{x_j} \cdot \overline{w} = (-1)^{d-i+1} \overline{x_{d+1}}\cdot \overline{w_j}$ in $\rr(J)$. Localizing at $\p$, $\overline{w_j}$ becomes a unit and so $(\overline{x_{d+1}})_\p \in (\overline{w})_\p$. Thus $\overline{\K}_\p = (\overline{w},\overline{x_{d+1}})_\p = (\overline{w})_\p$ and again the claim is satisfied. This combined with the result of \Cref{Kscm} shows the assumptions of \cite[3.4]{SV} are met and so indeed $\overline{\K}^m = \overline{\K}^{(m)}$.

The claim regarding the special fiber ring is now clear as $(x_1,\ldots,x_{d+1}) +\A = (x_1,\ldots,x_{d+1}) + (\det \BB_m)$ by \Cref{powersym}. Additionally note that $\det \BB_m \neq 0$ as it is the only equation pure in the variables $T_1,\ldots,T_{d+1}$ and $\ell(I)=d$. Hence modulo $(x_1,\ldots,x_{d+1})\rr(I)$, we see that $\ff(I)$ is indeed a hypersurface ring defined by an equation of degree $md$.
\end{proof}

\begin{corollary}\label[corollary]{other form of A}
In the setting of \Cref{mainresult}, the generating set of $\A= \BL_m + (\det \BB_m)$ is minimal. In particular, $\mu(\A) = d+m+1$ and $\mu(\J) = d+m$.
\end{corollary}

\begin{proof}
 First note that $\det \BB_i \neq 0$ for all $1\leq i\leq m$. Following the construction in \Cref{modit}, if $\det \BB_i =0$ for some $i$ in the range above then $\det \BB_j= 0$ for all $i\leq j\leq m$. However, this is impossible as $\det \BB_m \neq 0$ as mentioned in the proof of \Cref{mainresult}. Adopt the bigrading on $S[T_1,\ldots,T_{d+1}]$ given by $\deg x_i = (1,0)$ and $\deg T_i = (0,1)$. To show minimality, first note that $\BL_1= \LL$ is minimally generated by $\ell_1,\ldots,\ell_d,f$ as it is a complete intersection ideal of height $d+1$. With this, it suffices to show $\det \BB_i \notin \BL_{i}$ for all $1\leq i\leq m$ and we show this inductively. 
 
First suppose that $i=1$. If $m=1$, then $\det \BB_1$ is of bidegree $(0,d)$, hence $\det \BB_1\notin \BL_1=\LL$ by degree considerations. If $m\geq 2$, then $\deg (\det \BB_1) = (m-1,d)$ and so if $\det \BB_1 \in \BL_1 = \LL$, then $\det \BB_1\in (\ell_1,\ldots,\ell_d)$ as $\deg f =(m,0)$. However, in the creation of $\BB_2$, the column corresponding to $\det \BB_1$ is then a combination of the other $d$ columns of $B(\psi)$. Thus $\det \BB_2=0$ which is impossible as previously mentioned.

 Now suppose $m\geq i\geq 2$ and $\det \BB_j \notin \BL_{j}$ for all $1\leq j\leq i-1$. Notice that $\BL_i =(\ell_1,\ldots,\ell_d) + (f,\det \BB_1,\ldots,\det \BB_{i-1})$ and the bidegrees of these generators are known. Indeed, for all $1\leq j\leq m$ one has that $\deg (\det \BB_j) = (m-j,j\cdot d)$. Thus if $i< m$ and $\det \BB_i \in \BL_i$, by degree considerations again it must be that $\det \BB_i \in (\ell_1,\ldots,\ell_d)$. However, just as before in the creation of the next modified Jacobian dual iteration, this would imply $\det \BB_{i+1} =0$ which is impossible. In the case $i=m$, we see $\det \BB_m \notin \BL_m$ by degree reasons as $\deg (\det \BB_m) = (0,md)$. The claim regarding the number of generators of $\J$ then follows as $f$ is part of a minimal generating set of $\A$.
\end{proof}


We remark that the procedure of modified Jacobian dual iterations is quite simple when $\mu(I)=d+1$. In this case there are precisely $m$ iterations, where a single determinant is taken at each step allowing $\A$ to be built up, one minimal generator at a time. A similar procedure was implemented in \cite{CHW} and provided the inspiration for the algorithm of modified Jacobian dual iterations presented here.

\begin{example}
Let $S= k[x_1,x_2,x_3]$, $f=x_1^3$, and $R=S/(f)$. Consider the matrix with entries in $R$,
$$\varphi =\begin{bmatrix}
\overline{x_1} & \overline{x_3} \\
\overline{x_2}& \overline{x_1} \\
\overline{x_3}& \overline{x_2}\\
\end{bmatrix}$$
where $\overline{\,\cdot\,}$ denotes images modulo $(f)$. A simple computation shows $\grd I_2(\varphi) \geq 2$, hence the Hilbert-Burch theorem confirms the existence of a perfect $R$-ideal $I$ of grade $2$ with $\varphi$ as its presentation matrix. Additionally, $I$ satisfies the condition $G_2$ automatically. With this, we have
$$\psi =\begin{bmatrix}
x_1 & x_3 \\
x_2& x_1 \\
x_3& x_2\\
\end{bmatrix}$$
as the corresponding matrix of linear entries in $S$. This matrix has Jacobian dual
$$B(\psi) =\begin{bmatrix}
T_1 & T_2 \\
T_2& T_3 \\
T_3& T_1\\
\end{bmatrix}$$
and now we construct the modified Jacobian dual $B$ and perform modified Jacobian dual iterations. For any element $F\in S[T_1,T_2,T_3]$, let $\partial F$ denote any column of bihomogeneous entries of constant bidegree such that $[x_1\,\, x_2\,\, x_3] \cdot \partial F =F$. Recall the matrices in the method of modified Jacobian dual iterations are not necessarily unique, but this is addressed in \Cref{iterationssec}. In this example, the matrices can be taken as
\[
\begin{array}{lll}
      f=x_1^3,& 
       \BB_1 =[B(\psi)\,|\,\partial f]= \begin{bmatrix}
T_1 & T_2 &x_1^2\\
T_2& T_3&0 \\
T_3& T_1 &0\\
\end{bmatrix}\\[8ex]

        F_1 = \det \BB_1 = x_1^2(T_1T_2-T_3^2), & 
         \BB_2 =[B(\psi)\,|\,\partial F_1]= \begin{bmatrix}
T_1 & T_2 &x_1(T_1T_2-T_3^2)\\
T_2& T_3&0 \\
T_3& T_1 &0\\
\end{bmatrix}\\[8ex]

 F_2 = \det \BB_2 = x_1(T_1T_2-T_3^2)^2, & 
  \BB_3 =[B(\psi)\,|\,\partial F_2]= \begin{bmatrix}
T_1 & T_2 &(T_1T_2-T_3^2)^2\\
T_2& T_3&0 \\
T_3& T_1 &0\\
\end{bmatrix}\\[8ex]
 F_3 = \det \BB_3 = (T_1T_2-T_3^2)^3. 
\end{array}
\]
By \Cref{mainresult} we have $\A= \LL+(F_1,F_2,F_3)$ where $\LL = (\x\cdot B)$. Hence modulo $(f)$, the defining ideal of $\rr(I)$ is $\J = \overline{\LL+(F_1,F_2,F_3)}$. Notice that $\A$ and $\J$ are not prime ideals. This is to be expected as $\rr(I)$ is not a domain since $R$ is not a domain in this example.
\end{example}

\begin{remark}\label{differentials}
With the assumptions of \Cref{mainresult}, if $k$ is a field of characteristic zero or the characteristic of $k$ is strictly larger than $m$, there are two particular ways in which the equations of $\A$ can be described using differentials. 

\begin{enumerate}
    \item First, we may choose particular matrices in the process of modified Jacobian dual iterations. By setting $\BL_0 = (\ell_1,\ldots,\ell_d)$ and $F_0 = f$, one may recursively form triples
    $$\BL_i= \LL_{i-1} +(F_{i-1}) \quad \quad \BB_i = [B(\psi)\,|\, \partial F_{i-1}] \quad \quad F_i = \det \BB_i$$
    for $1\leq i \leq m$ where $\partial F_i = [\frac{\partial F_i}{\partial x_1}\ldots \frac{\partial F_i}{\partial x_{d+1}}]^t$, the column of partial derivatives of $F_i$. As $f$ is a homogeneous polynomial, each $F_i$ is as well. Thus it follows that the matrices $\BB_i$ above satisfy the conditions of \Cref{modit}. From the assumptions on the characteristic of $k$, each $\x$-degree of these equations is a unit, hence their multiples in the Euler formula do not affect ideal generation.
    
    \item In a similar manner, one may also use differential operators to describe the equations of $\A$. Let $\partial_{\x}$ denote the column $[\frac{\partial }{\partial x_1}\ldots \frac{\partial }{\partial x_{d+1}}]^t$ and consider the operator $\partial = \det [\B(\psi)\,|\, \partial_{\x}]$. Letting $\partial^i$ denote composition $i$ times, from the previous part, it follows that 
    $$\A = \big(\ell_1,\ldots,\ell_d,f,\partial(f),\partial^2(f),\ldots,\partial^m(f)\big)$$
    as the same equations are produced.
\end{enumerate}
Notice that in the first case, each $\BB_i$ is the transpose of the Jacobian matrix of $\ell_1,\ldots,\ell_d,F_{i-1}$ with respect to $x_1,\ldots,x_{d+1}$. The second method above is also interesting in that an algorithm is no longer necessary to produce the equations of $\A$, but rather a single differential operator. Using $D$-modules, a similar approach was taken to study the defining ideal of the Rees algebra in \cite{CR}. 

It is also interesting to note that these alternative descriptions provide an additional reason as to why there are exactly $m$ iterations of the modified Jacobian dual. As $B(\psi)$ consists of entries in $k[T_1,\ldots,T_{d+1}]$, the $\x$-degree of each equation produced decreases by $1$ in each step after differentiating and taking a determinant. Hence the procedure must terminate precisely after $m$ steps.
\end{remark}

\begin{remark}\label[remark]{f linear case}
We note that \Cref{mainresult} recovers the main result of \cite{MU} (in the case $\mu(I)=d+1$) when $m=1$. After a change of coordinates, it can be assumed that the factored equation $f$ is one of the indeterminates, say $f=x_{d+1}$. Thus $R\cong k[x_1,\ldots,x_d]$ and we remark that the submatrix $B'$ of $B(\psi)$, as in \Cref{notation2}, is exactly the Jacobian dual of the presentation matrix of $I$ in this ring with respect to $x_1,\ldots,x_{d}$. Now the column corresponding to $f$ in the modified Jacobian dual $B$ consists of all zeros except for a 1 in the last entry. Thus the determinant of $B$, the first and only iteration, is exactly the determinant of $B'$.
\end{remark}

\subsection{Depth and Cohen-Macaulayness}\label[section]{depthsec}

With the assumptions of \Cref{mainresult}, we study the depth and Cohen-Macaulay property of $\rr(I)$ now that the defining equations are understood. Despite $\J$ being the defining ideal of $\rr(I)$ in the traditional sense, it will be more convenient to use the $S[T_1,\ldots,T_{d+1}]$-ideal $\A$ and the isomorphism $\rr(I) \cong S[T_1,\ldots,T_{d+1}]/\A$. We begin by creating a handful of short exact sequences which will be essential to our study. We adopt the setting of \Cref{notation2} throughout.

Let $\m =(x_1,\ldots,x_{d+1})$ and recall from \Cref{Kscm}, $\overline{\K} = (\overline{w},\overline{x_{d+1}})$ where $w= \det B'$. Recall $\overline{\K}$ is a Cohen-Macaulay ideal and by \Cref{colons}, $\m \rr(J) = (\overline{x_{d+1}}):\overline{\K}$. Hence there is short exact sequence of bigraded $\rr(J)$-modules
\[
0\longrightarrow \m \rr(J) (0,-d) \longrightarrow \rr(J)(-1,0)\oplus \rr(J)(0,-d)\longrightarrow \overline{\K} \longrightarrow 0.
\]
Passing to the induced sequence obtained by applying the functor $\sym(-)$ and considering the $m^{\text{th}}$ graded strand, we obtain 
$$\m \rr(J) (0,-d) \otimes \text{Sym}_{m-1}\big(\rr(J)(-1,0)\oplus \rr(J)(0,-d)\big)\overset{\sigma}{\longrightarrow}\hspace{40mm}$$ 
$$\hspace{40mm}\text{Sym}_{m}\big(\rr(J)(-1,0)\oplus \rr(J)(0,-d)\big) \longrightarrow \text{Sym}_m(\overline{\K}) \longrightarrow 0.$$

Notice that $\ker \sigma$ is torsion due to rank considerations. However, it is a submodule of a torsion-free $\rr(J)$-module, hence it must be zero. Thus $\sigma$ is injective and we now have the short exact sequence
$$0\longrightarrow \m \rr(J) (0,-d) \otimes \text{Sym}_{m-1}\big(\rr(J)(-1,0)\oplus \rr(J)(0,-d)\big)\overset{\sigma}{\longrightarrow}\hspace{30mm}$$ 
$$\hspace{40mm}\text{Sym}_{m}\big(\rr(J)(-1,0)\oplus \rr(J)(0,-d)\big) \longrightarrow \text{Sym}_m(\overline{\K}) \longrightarrow 0.$$

From \Cref{Kscm}, it can be seen that $\overline{\K}$ satisfies the $G_\infty$ condition. As \Cref{Kscm} also shows $\overline{\K}$ is strongly Cohen-Macaulay, by \cite[2.6]{HSV1} it is an ideal of linear type, hence $\sym_m(\overline{\K}) \cong \overline{\K}^m$. With this and passing to a direct sum decomposition, the short exact sequence above is
\begin{equation}\label[equation]{directsumseq}
0\longrightarrow \displaystyle{\bigoplus_{i=0}^{m-1}}\,\m \rr(J)\big(-i,-(m-i)d\big)\longrightarrow \displaystyle{\bigoplus_{i=0}^{m}}\, \rr(J)\big(-i,-(m-i)d\big) \longrightarrow \overline{\K}^m\longrightarrow 0.
\end{equation}

We are now ready to compute the depth of $\rr(I)$. Recall a Noetherian local ring $A$ is said to be \textit{almost} Cohen-Macaulay if $\dep A = \dim A-1$.

\begin{theorem}\label{depth}
In the setting of \Cref{mainresult}, $\rr(I)$ is Cohen-Macaulay if and only if $m=1$ and is almost Cohen-Macaulay otherwise. Additionally, $\ff(I)$ is Cohen-Macaulay.
\end{theorem}

\begin{proof}
From the short exact sequence
\begin{equation}\label{regseq1}
0 \longrightarrow\m \rr(J) \longrightarrow \rr(J) \longrightarrow \ff(J)\longrightarrow 0
\end{equation}
it follows that $\dep \m \rr(J) \geq d+2$, hence $\dep \m \rr(J) =d+2$. This together with (\ref{directsumseq}) shows that $\dep \overline{\K}^m \geq d+1$. Finally, the sequence
\begin{equation}\label{regseq2}
0 \longrightarrow \overline{\A} \longrightarrow \rr(J) \longrightarrow  \rr(I) \longrightarrow 0
\end{equation}
and the isomorphism $\overline{\A} = \frac{\overline{f}\overline{\K}^{(m)}}{\overline{x_{d+1}}^m} \cong \overline{\K}^m$ give that $\dep \rr(I) \geq d$. The Cohen-Macaulayness in the case $m=1$ follows from \Cref{f linear case} and \cite{MU}. Now if $m\geq 2$, it can be seen that $\rr(I)$ is not Cohen-Macaulay by either \cite[4.5]{SUV2} or \cite[2.1]{PU}. Thus in this latter case we have $\dep \rr(I) =d$, hence $\rr(I)$ is indeed almost Cohen-Macaulay. The assertion on the Cohen-Macaulayness of $\ff(I)$ is clear as it is a hypersurface ring by \Cref{mainresult}.
\end{proof}

\subsection{Relation Type and Regularity}

We now introduce two important numerical invariants, the relation type and the Castelnuovo-Mumford regularity. The \textit{relation type} $\rt(I)$ of $I$ is simply the maximum $\T$-degree appearing in a minimal generating set of the defining ideal of the Rees algebra $\rr(I)$. 

For the regularity, we follow the definitions and conventions of \cite{Trung}. Let $A= \bigoplus_{n\geq 0} A_n$ be a finitely generated standard graded ring over Noetherian ring $A_0$. For a nonzero graded $A$-module $M$, we define
$$a(M) = \max\big\{ n \,\big| \, M_n \neq 0\big\}$$
where $M_n$ denotes the homogeneous degree $n$ component of $M$. Furthermore, we adopt the convention that $a(M) = -\infty$ if $M=0$. With this, for $i\geq 0$ define $a_i(A) =a\big( H_{A_+}^i(A)\big)$ where $A_+$ is the $A$-ideal generated by homogeneous elements of positive degree and $H_{A_+}^i(-)$ is the $i^{\text{th}}$ local cohomology functor with respect to this ideal. The \textit{Castelnuovo-Mumford regularity} of $A$ is defined as
$$\reg(A) = \max\{a_i(A) + i \,|\, i\geq 0\}.$$

As there are multiple gradings on $\rr(I)$, we consider its regularity with respect to $\m = (x_1,\ldots,x_{d+1})$, $\tt=(T_1,\ldots,T_{d+1})$, and $\n=(x_1,\ldots,x_{d+1},T_1,\ldots,T_{d+1})$. When computing regularity with respect to $\m$ we set $\deg x_i=1$ and $\deg T_i =0$. Similarly, when computing regularity with respect to $\tt$ we set $\deg x_i=0$ and $\deg T_i =1$. Lastly, when computing regularity with respect to $\n$, we adopt the total grading and set $\deg x_i=1$ and $\deg T_i =1$.

\begin{theorem}\label[theorem]{rtandreg}
In the setting of \Cref{mainresult}, we have
$$\rt(I) = \reg \ff(I) +1 = {\rm reg}_\tt \rr(I)+1 = md.$$
Additionally, $\reg_\m \rr(I) \leq m-1$ and $\reg_\n \rr(I) \leq (m+1)d$.
\end{theorem}

\begin{proof}
The statement regarding the relation type follows immediately from \Cref{mainresult}. Additionally, the claim that $\reg \ff(I) = md-1$ is clear as $\ff(I) \cong k[T_1,\ldots,T_{d+1}]/(\mathfrak{f})$ where $\mathfrak{f}$ is of degree $md$. For the regularity of $\rr(I)$ with respect to $\tt$, it is well known that $\rt(I) -1 \leq \reg_\tt \rr(I)$ so it suffices to show $\reg_\tt \rr(I) \leq md-1$ and we show this and the other inequalities simultaneously.
We use (\ref{regseq1}) and (\ref{regseq2}) once more, noting that Castelnuovo-Mumford regularity is comparable on short exact sequences \cite{Eisenbud}. First note that as $J$ is of linear type, $\rr(J)$ is a complete intersection domain defined by forms linear in both $x_1,\ldots,x_{d+1}$ and $T_1,\ldots,T_{d+1}$. Additionally, note that $\ff(J) =\rr(J)/\m \rr(J)  \cong k[T_1,\ldots,T_{d+1}]$. With this, we have
\[
\begin{array}{lcl}
   {\rm reg}_\tt \rr(J) = 0,  &\quad &{\rm reg}_\tt \ff(J) = 0 \\[1ex]
     {\rm reg}_\m \rr(J) = 0,  & \quad &{\rm reg}_\m \ff(J) = 0\\[1ex]
     {\rm reg}_\n \rr(J) = d, &\quad  &{\rm reg}_\n \ff(J) = 0.\\
     \end{array}
\]
Thus from (\ref{regseq1}) we have 
\[
\begin{array}{ccc}
     {\rm reg}_\tt\, \m \rr(J)  \leq 1, &  {\rm reg}_\m \,\m \rr(J)  \leq 1, & {\rm reg}_\n\, \m \rr(J)  =d.\\
\end{array}
\]
For convenience write
\[
\begin{array}{ccc}
    M= \displaystyle{\bigoplus_{i=0}^{m-1}\m \rr(J)\big(-i,-(m-i)d\big)}, & \quad &  N= \displaystyle{\bigoplus_{i=0}^{m} \rr(J)\big(-i,-(m-i)d\big).}   \\ 
\end{array}
\]
With this and the above, we obtain
\[
\begin{array}{lcl}
   {\rm reg}_\tt M \leq md+1,  &\quad &{\rm reg}_\tt N = md \\[1ex]
     {\rm reg}_\m M \leq m,  & \quad &{\rm reg}_\m N = m\\[1ex]
     {\rm reg}_\n M \leq (m+1)d, &\quad  &{\rm reg}_\n N = (m+1)d.\\
     \end{array}
\]
Now using (\ref{directsumseq}), we have 
\[
\begin{array}{ccc}
     {\rm reg}_\tt \overline{\K}^m  \leq md, &  {\rm reg}_\m \overline{\K}^m  \leq m, & {\rm reg}_\n \overline{\K}^m \leq (m+1)d.\\
\end{array}
\]
Lastly, the inequalities above, the
bigraded isomorphism $\overline{\A} \cong \overline{\K}^m(0,-1)$, and the sequence (\ref{regseq2}) give
\[
\begin{array}{ccc}
     {\rm reg}_\tt \rr(I)  \leq md-1, &  {\rm reg}_\m \rr(I)  \leq m-1, & {\rm reg}_\n \rr(I) \leq (m+1)d-1 .  
\end{array}\qedhere
\]
\end{proof}

\section{Rees Algebras of Modules over Hypersurface Rings}\label[section]{modulesec}

We now consider Rees algebras of modules over hypersurface rings. For such a module $E$, we introduce a generic Bourbaki ideal $I$ which,  with some additional assumptions, is an ideal of height 2 in a hypersurface ring. We then relate the study of $\rr(E)$ to the study of $\rr(I)$, making use of the results from the previous section.

As the connection between almost linear presentation within polynomial rings and linear presentation within hypersurface rings has been established for ideals, we extend this to modules by following the approach of Section 5 of \cite{Costantini}. We take all conventions and notation from \cite{ReesAlgebrasOfModules} restated in \Cref{prelims} along with the construction of the generic Bourbaki ideal.

\subsection{Rees Algebras of Modules}

As before in the case of Rees algebras of ideals, the difficulty in the study of defining ideals of Rees algebras of modules is determining the nonlinear equations. To describe these equations, we produce constructions similar to those in \Cref{iterationssec}. Throughout this section we consider the situation below.

\begin{setting}\label[setting]{module setting}
Let $S =k[x_1,\ldots,x_{d+1}]$ be a polynomial ring over a field $k$ with $d\geq 2$, $f\in S$ a homogeneous polynomial of degree $m\geq 1$, and $R=S/(f)$. Let $E$ be a finite $R$-module minimally generated by $n$ homogeneous elements of the same degree. Further assume that $E$ has projective dimension one and hence a minimal free resolution of the form
\[
0 \longrightarrow R^{n-e}\overset{\varphi}{\longrightarrow} R^n \longrightarrow E\longrightarrow 0
\]
where $e$ denotes the rank of $E$. Lastly, assume that $\varphi$ consists of homogeneous linear entries in $R$. 
\end{setting}

After localizing at the homogeneous maximal ideal, $E$ admits a generic Bourbaki ideal $I$ which is necessarily perfect of grade $2$. Following the notation of \Cref{prelims}, we remark that $I$ is an ideal of $R''$ which is a hypersurface ring as $R$ is. Recall from \Cref{hypring} that for such an ideal $I$, it was imperative to return to a polynomial ring to study $\rr(I)$. Hence we introduce notation to permit this in the study of $\rr(E)$.

\begin{notation}
With $R$ and $S$ as in \Cref{module setting} and $Z$ the set of indeterminates from \Cref{module notation}, let $S'=S[Z]$ and $S'' = S(Z) = S'_{(\x)S'}$ and notice that $R' = S'/(f)S'$ and $R''=S''/(f)S''$. In an abuse of notation, let $\overline{\,\cdot\,}$ denote images modulo $(f)$ in $R$, $R'$, and $R''$ as it will be clear from context which ring is being considered. Additionally, let $\psi$ be an $n\times n-e$ matrix of linear entries in $S$ with $\varphi = \overline{\psi}$.
\end{notation}

Just as before, we have the notion of a modified Jacobian dual matrix. With $\psi$ as above, let $B(\psi)$ denote its Jacobian dual with respect to $\x=x_1,\ldots,x_{d+1}$. If $\partial f$ is any column matrix with entries in $S$ such that $\x \cdot \partial f =f$, let $B = [B(\psi)\,|\,\partial f]$ be a modified Jacobian dual of $\psi$ and set $\LL = (\x\cdot B)$. With this, we are able to perform both of the algorithms presented in \Cref{iterationssec}. Let $\LL+I_{d+1}(B_i)$ denote the ideal of matrix iterations with respect to $x_1,\ldots,x_{d+1}$ obtained by applying \Cref{defit} to $B$. Moreover, let $\BL_i +I_{d,1}(B(\psi),\BB_i)$ denote the ideal obtained by applying \Cref{modit} to the pair $(\LL,B)$. Once more, we refer to this second algorithm as the method of modified Jacobian dual iterations.

\begin{theorem}\label[theorem]{mainresultformodules}
With the assumptions of \Cref{module setting}, further assume that $E$ satisfies $G_d$, $n =d+e$, and $I_1(\varphi) = \overline{(x_1,\ldots,x_{d+1})}$. The defining ideal $\J$ of $\rr(E)$ satisfies 
$$\J =\overline{\LL+I_{d+1}(B_m)} = \overline{\BL_m + (\det \BB_m)}.$$
Additionally, $\rr(E)$ is Cohen-Macaulay if and only if $m=1$ and is almost Cohen-Macaulay otherwise.
\end{theorem}

\begin{proof}
First note that if $e=1$, then $E$ is isomorphic to a perfect $R$-ideal of grade 2 and the result follows from \Cref{mainresult}, hence we may assume $e\geq 2$. Let $a_1,\ldots, a_n$ denote a minimal generating set of $E$ corresponding to $\varphi$ and consider the natural epimorphism $R[T_1,\ldots,T_n] \rightarrow \rr(E)$ mapping $T_i \mapsto a_i\in [\rr(E)]_1$ for $1\leq i\leq n$. After replacing $S$ and $R$ by their localizations at their respective homogeneous maximal ideals, we may assume that $S$ and $R$ are local rings and $E$ admits a generic Bourbaki ideal $I$. With this, $I$ is necessarily perfect of grade 2 and $\mu(I) = n-e+1 = d+1$.

Now with $y_j$ as in \Cref{module notation}, let $Y_j = \sum_{i=1}^n Z_{ij} T_i$ for $1\leq j\leq e-1$ and note that $Y_j$ maps to $y_j$ under the natural epimorphism $R''[T_1,\ldots,T_n] \rightarrow \rr(E'')$. By \Cref{gbipresentation}, there is a minimal presentation $\varphi_I$ of $I$ such that 
$$\x\cdot B(\varphi)  \equiv  \T \cdot \begin{bmatrix} \hspace{2mm}0\hspace{2mm} \\
\varphi_I\end{bmatrix} \mod (Y_1,\ldots,Y_{e-1}).$$
Now there exists a matrix $\psi_I$ with linear entries in $S''$ such that $\overline{\psi_I} =\varphi_I$ and  
$$\x\cdot B(\psi)  \equiv  \T \cdot \begin{bmatrix} \hspace{2mm}0\hspace{2mm} \\
\psi_I \end{bmatrix} \mod (Y_1,\ldots,Y_{e-1}).$$
Now with this matrix $\psi_I$, we construct the modified Jacobian dual of $\psi_I$ and use \Cref{mainresult} to describe the defining ideal of $\rr(I)$. As a matter of notation we continue to write $f$ for its image in $S''$ and $\partial f$ for a column such that $\x \cdot \partial f =f$. After choosing such a column, we take $B_I = [B(\psi_I)\,|\,\partial f]$ to be a modified Jacobian dual of $\psi_I$, hence by \Cref{mainresult} the defining ideal of $\rr(I)$ is the $R''[T_1,\ldots,T_n]$-ideal 
$$\J_I =\overline{\BL_{I,m} + (\det\BB_{I,m})}$$
 where $(\BB_{I,m},\BL_{I,m})$ is the $m^{\text{th}}$ modified Jacobian dual iteration of $(B_I,\LL_I)$ for $\LL_I=(\x\cdot B_I)$. Additionally, $\rr(I)$ is Cohen-Macaulay if and only if $m=1$ and is almost Cohen-Macaulay otherwise. Thus by \Cref{passCMness}, $\rr(E)$ is Cohen-Macaulay if and only if $m=1$.

As $E''$ is of projective dimension one and satisfies $G_d$, by either \cite[3 and 4]{Avramov} or \cite[1.1]{Huneke}, it follows that $E''$ is of linear type on the punctured spectrum of $R''$, hence $I$ is as well. Additionally, note that $\dep \rr(I) \geq \dim \rr(I) -1 = d\geq 2$. Thus by induction on $e\geq 2$ and using \cite[3.1]{Costantini} repeatedly, we find that $\rr(I) \cong \rr(E'')/(F'')$ and $y_1,\ldots,y_{e-1}$ forms a regular sequence on $\rr(E'')$. Thus $Y_1,\ldots,Y_{e-1}$ forms a regular sequence modulo $\J R''[T_1,\ldots,T_{d+e}]$. Hence $\rr(E'')$ is almost Cohen-Macaulay when $m\geq 2$ and with this, $\rr(E)$ is almost Cohen-Macaulay when $m\geq 2$. This also shows
$$\J_I = \J R''[T_1,\ldots,T_{d+e}] + (Y_1,\ldots,Y_{e-1}).$$
Now with this, \Cref{mainresult}, and \Cref{jacobianduallemma} we have
$$\J R''[T_1,\ldots,T_{d+e}] +(Y_1,\ldots,Y_{e-1}) = \overline{\BL_m + (\det \BB_m)}+(Y_1,\ldots,Y_{e-1}).$$
Now $E$ is of linear type on the punctured spectrum of $R$, from which it follows that 
\begin{equation}\label{contain}
\overline{\BL_m + (\det \BB_m)} \subseteq \overline{\LL + I_{d+1}(B_m)} \subseteq \overline{\LL:(x_1,\ldots,x_{d+1})^m} \subseteq \J.
\end{equation}
With this and the fact that $Y_1,\ldots,Y_{e-1}$ forms a regular sequence modulo the ideal $\J R''[T_1,\ldots,T_{d+e}]$, we have
$$\J =  \big( \,\overline{\BL_m + (\det \BB_m)} +(Y_1,\ldots, Y_{e-1}) \big) \cap \J = \overline{\BL_m + (\det \BB_m)} +(Y_1,\ldots, Y_{e-1}) \J.$$
Hence by Nakayama's lemma, we have that 
$$\J = \overline{\BL_m + (\det \BB_m)}.$$
This equality and (\ref{contain}) show that $\J=\overline{\LL + I_{d+1}(B_m)}$ as well.
\end{proof}


\begin{lemma}\label[lemma]{jacobianduallemma}
Adopt the setting and notation of \Cref{mainresultformodules}. Letting $B$ and $B_I$ denote the respective modified Jacobian dual matrices with $\LL =(\x\cdot B)$ and $\LL_I =(\x\cdot B_I)$, we have
$$\BL_m  +(\det \BB_m) +(Y_1,\ldots,Y_{e-1}) = \BL_{I,m}  +(\det \BB_{I,m})$$
where $(\BB_m,\BL_m)$ and $(\BB_{I,m},\BL_{I,m})$ are the modified Jacobian dual iterations of the pairs $(B,\LL)$ and $(B_I,\LL_I)$ respectively.
\end{lemma}

\begin{proof}
We show $\BL_i  +(\det \BB_i) +(Y_1,\ldots,Y_{e-1}) = \BL_{I,i}  +(\det \BB_{I,i})$ for all $1\leq i\leq m$ by induction on $i$. With $B(\psi_I)$ the Jacobian dual of $\psi_I$, we have 
$$\x\cdot B(\psi_I)  =  \T \cdot \begin{bmatrix} \hspace{2mm}0\hspace{2mm} \\
\psi_I \end{bmatrix}$$
as in the proof of \Cref{mainresultformodules}. Now in $S''[T_1,\ldots,T_{d+e}]$, we see 
$$\x\cdot B  \equiv  \x \cdot B_I \mod (Y_1,\ldots,Y_{e-1})$$
and so the statement is proved for $i=1$ as $\BL_1=(\x\cdot B)$, $\BL_{I,1}=(\x\cdot B_I)$, $\BB_1=B$, and $\BB_{I,1}=B_I$. Now suppose that $m\geq 2$ and the statement holds for all $i$ with $1\leq i\leq m-1$. Following the definition of $\BB_m$, let
$$\BL_{m-1} + (\det\BB_{m-1}) = \BL_{m-1} + (\x\cdot C)$$
for a column $C$ as in \Cref{modit}. Since the matrices $\BB_i$ are bigraded, modulo $(Y_1,\ldots, Y_{e-1})$ we have 
$$\BL_{I,m-1} + (\det \BB_{I,m-1}) = \BL_{I,m-1} + (\x \cdot C')$$
following the induction hypothesis, where $C'$ denotes the image of $C$ modulo the ideal $(Y_1,\ldots,Y_{e-1})$. Now take $\BB_{I,m}= [B(\psi)\,|\,C]$ and recall that $\BL_m =\BL_{m-1} + (\det\BB_{m-1})$ and $\BL_{I,m}=\BL_{I,m-1} + (\det \BB_{I,m-1})$. With this, in $S''[T_1,\ldots,T_{d+e}]$ we have
$$\BL_m  +(\det \BB_m) +(Y_1,\ldots,Y_{e-1}) = \BL_{I,m}  +(\det \BB_{I,m})$$
as claimed.
\end{proof}

\begin{remark}
Notice that the result of \cite[4.11]{ReesAlgebrasOfModules} (in the case $\mu(E)=d+e$) is recovered when $m=1$. After a change of coordinates in may be assumed that $f=x_{d+1}$ and so $R\cong k[x_1,\ldots,x_d]$. The result then follows in a similar manner to \Cref{f linear case}.
\end{remark}

\begin{remark}\label{differantialsformodules}
 When $k$ is a field of characteristic zero or the characteristic of $k$ is strictly larger than $m$, the equations defining $\rr(E)$ can be described using differentials once more as in \Cref{differentials}.
\end{remark}

\begin{corollary}\label[corollary]{CMnessofF(E)}
In the setting of \Cref{mainresultformodules}, $\ff(I) \cong k[T_1,\ldots,T_{d+e}]/(\mathfrak{f})$ where $\deg \mathfrak{f} = md$. In particular, $\ff(E)$ is Cohen-Macaulay.
\end{corollary}

\begin{proof}
By \Cref{mainresultformodules}, $\J =\overline{\BL_m + (\det \BB_m)}$ and we see the only equation of $\J$ pure in the variables $T_1,\ldots,T_{d+e}$ is $\overline{\det \BB_m}$. Checking degrees similarly as in the proof of \Cref{other form of A}, we see that $\ff(E)$ is indeed a hypersurface ring defined by an equation of degree $md$.
\end{proof}

\begin{remark}
 A similar argument to the proof of \Cref{other form of A} shows that $\J = \overline{\BL_m + (\det \BB_m)}$ is minimally generated and $\mu(\J)= d+m$. Additionally, by checking degrees as before, the relation type is seen to be $\rt(E) = md$.
\end{remark}

In the proof of \Cref{mainresultformodules}, it was shown that $\rr(E'')$ is a deformation of $\rr(I)$. With this we may investigate the regularity of $\rr(E)$. As before, let $\m = (x_1,\ldots,x_{d+1})$, $\tt=(T_1,\ldots,T_{d+1})$, and $\n=(x_1,\ldots,x_{d+1},T_1,\ldots,T_{d+1})$ and consider the three different gradings as in \Cref{rtandreg}.

\begin{theorem}\label[theorem]{regofR(E)}
In the setting of \Cref{mainresultformodules}, ${\rm reg}_\tt \rr(E) = md-1$, $\reg_\m \rr(E) \leq m-1$, and $\reg_\n \rr(E) \leq (m+1)d$. Additionally, ${\rm reg} \ff(E) =md-1$.
\end{theorem}

\begin{proof}
From the proof \Cref{mainresultformodules}, there is an isomorphism $\rr(I)\cong \rr(E'')/(F'')$. As $F''$ is a free $R''$-module generated by forms of bidegree $(0,1)$, it follows that ${\rm reg}_\tt \rr(E'') = {\rm reg}_\tt \rr(I)$ and $\reg_\n \rr(E'') =\reg_\n \rr(I)$. The short exact sequence
$$0 \longrightarrow (F'')\longrightarrow \rr(E'') \longrightarrow \rr(I) \longrightarrow 0$$
shows $\reg_\m \rr(E'')\leq \reg_\m \rr(I)$. Now using \Cref{rtandreg}, we see
\[
\begin{array}{ccc}
     {\rm reg}_\tt \rr(E'')  = md-1, &  {\rm reg}_\m \rr(E'')  \leq m-1, & {\rm reg}_\n \rr(E'') \leq (m+1)d-1  \\
\end{array}
\]
and we note that regularity is unchanged when passing from $\rr(E)$ to $\rr(E'')$. Lastly, the regularity of the special fiber ring is clear as $\ff(E)$ is a hypersurface ring defined by an equation of degree $md$.
\end{proof}

\section*{Acknowledgements}

The author would like to thank his Ph.D. advisor Bernd Ulrich for many insightful discussions and comments on the results presented here, without which this work would not have been possible. Additionally, the author would like to thank Alessandra Costantini for many enlightening conversations and helpful suggestions allowing the results in \Cref{modulesec} to come to fruition.


\end{document}